\documentclass[11pt,reqno]{amsart}
\usepackage{amsmath}    
\usepackage{amssymb}    
\usepackage{amsthm}     
\usepackage{amscd}      
\usepackage{mathtools}  

\usepackage{dsfont}     
\usepackage{latexsym}   
\usepackage{mathrsfs} 
\usepackage{bm} 

\usepackage{algorithm}
\usepackage{algorithmic}

\usepackage{tikz}           
\usetikzlibrary{            
  positioning,              
  decorations.pathreplacing 
}
\usepackage[all]{xy}        

\usepackage[pdfstartview=FitH, bookmarksnumbered=true,bookmarksopen=true, colorlinks=true, pdfborder={0 0 1}, citecolor=blue, linkcolor=blue,urlcolor=blue]{hyperref}

\usepackage[square,comma,sort&compress,numbers]{natbib}
\renewcommand{\cite}{\citet*} 
\setcitestyle{authoryear,round}

\usepackage{graphicx}      
\usepackage{float}         

\usepackage{caption}       
\usepackage{subcaption}    
\usepackage{xcolor}        

\usepackage[margin=2.9cm,  bottom=3.1cm]{geometry}

\newcommand{\dd}{\mathrm{d}}

\newcommand{\E}{\mathbb{E}}         
\newcommand{\EE}{\mathbb{E}}
\newcommand{\mP}{\mathbb{P}}
\newcommand{\mR}{\mathbb{R}}        

\newcommand{\mb}[1]{\mathbb{#1}}
\newcommand{\mbf}[1]{\mathbf{#1}}
\newcommand{\mc}[1]{\mathcal{#1}}

\newcommand{\pt}{\partial}
\newcommand{\nm}[1]{\left\| #1 \right\|} 

\theoremstyle{plain}
\newtheorem{theorem}{Theorem}[section]

\newtheorem{lemma}{Lemma}[section]
\newtheorem{proposition}{Proposition}[section]

\theoremstyle{definition}
\newtheorem{definition}{Definition}[section]
\newtheorem{example}{Example}[section]

\newtheorem{assumption}{Assumption}[section]

\theoremstyle{remark}
\newtheorem{remark}{Remark}[section] 

\theoremstyle{definition}

\numberwithin{equation}{section} 

\numberwithin{equation}{section}

\begin{document}
\makeatletter
\def\@setauthors{%
\begingroup
\def\thanks{\protect\thanks@warning}%
\trivlist \centering\footnotesize \@topsep30\p@\relax
\advance\@topsep by -\baselineskip
\item\relax
\author@andify\authors
\def\\{\protect\linebreak}%
{\authors}%
\ifx\@empty\contribs \else ,\penalty-3 \space \@setcontribs
\@closetoccontribs \fi
\endtrivlist
\endgroup } \makeatother
 \baselineskip 18pt
\title[{ \tiny Utility Maximization under Distributional Ambiguity}]
{{
		Robust Utility Maximization with Intractable Claims under Distributional Ambiguity: A Random Distributionally Robust Optimization Approach }} \vskip 10pt\noindent
\author[{\tiny  Guohui Guan, Zongxia Liang, Xingjian Ma}]
{\tiny {\tiny  Guohui Guan$^{a,b,\dag}$, Zongxia Liang$^{c,\ddag}$, Xingjian Ma$^{c,*}$}}

\noindent
\begin{abstract}
{~\\}
This paper studies a robust utility maximization problem for intractable claims under distributional ambiguity, where the distribution of the claim cannot be inferred from market information and its dependence with tradable assets is largely unknown. 
We extend the existing framework for intractable claims in two directions. First, we allow the marginal distribution of the claim to vary within a $\varphi$-divergence ambiguity set, capturing statistical uncertainty in its estimation. Second, we consider a general (possibly non-additive) bivariate utility function, which enables more flexible interactions between the decision and the claim beyond the classical additive specification. 
To analyze this problem, we adopt a random distributionally robust optimization (RDRO) formulation, which lifts the optimization to the space of joint distributions and provides a convenient representation of the coupling between the decision and the uncertain claim. We establish the existence of optimal decisions using tools from optimal transport and develop a Legendre-Fenchel duality framework that links the constrained and penalized formulations, leading to uniqueness results and tractable reformulations. 
Finally, we propose a numerical algorithm based on unbalanced optimal transport scaling combined with projected gradient methods, and illustrate the relationship between the parameters in the constrained and penalized formulations.
\vskip 5pt  \noindent
\vskip 5pt  \noindent
MSCcodes: 90C15, 90C46, 90C47, 91G10
\end{abstract}

\maketitle

\vskip15pt
\setcounter{equation}{0}

\section{ \bf Introduction}

The expected utility framework is a cornerstone of modern portfolio selection and asset pricing theory. In its classical form, an investor selects a terminal payoff or trading strategy so as to maximize expected utility under a reference probability measure. In such formulations, all sources of uncertainty are assumed to be fully specified through a given probabilistic model, and the distribution of the underlying risks is taken as known (see, e.g., \cite{merton1975optimum}, \cite{karatzas1998methods}).

In many applications, however, this assumption is often  unrealistic. A prominent example arises in the presence of \textbf{intractable contingent claims}, i.e., contingent claims whose distributions cannot be inferred from market information and whose dependence with tradable assets is largely unknown. Typical examples include insurance liabilities, employee compensation schemes, or other non-tradable exposures (see, e.g., \cite{hou2016robust}, \cite{li2023robust}). In such settings, the investor typically has access only to partial distributional information, most notably, an estimate of the marginal distribution of the claim, while the joint structure with financial positions remains unspecified.

Motivated by these considerations, recent works have proposed robust utility maximization models in which the dependence between the terminal wealth and the intractable claim is treated in a worst-case manner; see, for instance, \cite{hou2016robust} and \cite{li2023robust}. A key feature of these models is that the marginal distribution of the claim is assumed to be fixed. This structure allows the problem to be reformulated using quantile-based methods or rearrangement techniques, as further developed in recent work such as \cite{chen2026alpha}.

However, such approaches rely critically on the knowledge of the marginal distribution and its associated quantile representation. When the marginal distribution is itself uncertain, these techniques are no longer directly applicable, since the quantile representation depends on the underlying marginal distribution, which is no longer fixed in our setting. This motivates the need for a framework that can accommodate both dependence uncertainty and marginal ambiguity.

\subsection{Problem setup}

In this paper, we extend the above framework by incorporating distributional ambiguity at the marginal level, while preserving the worst-case treatment of dependence. From a modeling perspective, allowing the marginal distribution to vary naturally connects the problem to distributionally robust optimization (DRO), where uncertainty is described through a set of plausible probability measures (see, e.g., \cite{delage2010distributionally}, \cite{goh2010distributionally},  \cite{ben2013robust}, \cite{rahimian2022frameworks}, \cite{jiang2024data}). At the same time, the worst-case treatment of dependence leads to an optimization over couplings, which is closely related to optimal transport (see, e.g., \cite{villani2008optimal}).

More precisely, we consider the robust utility maximization problem
\[
\sup_{X \in \mc{A}} \inf_{\nu \in D_\eta} \inf_{Y \sim \nu} \mathbb{E}[U(X,Y)],
\]
where $D_\eta$ denotes an ambiguity set defined through a $\varphi$-divergence neighborhood of a nominal distribution with tolerable error $\eta$ (see, e.g., \cite{ben2013robust}, \cite{love2015phi}). Such ambiguity sets are widely used due to their tractability and their compatibility with convex duality methods. 

This formulation captures two sources of uncertainty: marginal ambiguity, through $D_\eta$, and dependence uncertainty, through the coupling between the decision $X$ and the intractable claim $Y$. Although these two sources are conceptually distinct, we express the problem using a single infimum for notational simplicity.

In addition, we allow for a general (possibly non-additive) bivariate utility function $U(X,Y)$, which extends the classical additive form $U(X+Y)$ used in the existing literature. This extension enhances modeling flexibility and allows for more general interactions between the decision and the intractable claim, which are not captured by the classical additive specification.

\subsection{Challenges and contributions}

The simultaneous presence of marginal ambiguity and dependence uncertainty leads to a substantially more challenging problem than existing models. In particular, the ambiguity is imposed on the marginal distribution, while the optimization is naturally carried out over joint distributions. This feature is not covered by standard approaches in distributionally robust optimization, where ambiguity is typically imposed on the full distribution, nor by classical optimal transport, where both marginals are fixed.

To analyze this problem, we adopt a random distributionally robust optimization (RDRO) approach, in which the decision variable is modeled as a random variable and the optimization is lifted to the space of joint distributions. This representation makes it possible to describe the coupling between the decision and the uncertain claim and to use tools from optimal transport. The RDRO formulation is used here primarily as a technical device that enables a tractable analysis of the problem under joint ambiguity.

The proposed framework raises several analytical challenges. First, the optimization is carried out over an infinite-dimensional space, and the existence of optimal solutions requires appropriate compactness and stability arguments, relying on tools from optimal transport. Second, the ambiguity constraint acts on the marginal distribution rather than on the joint law, which prevents a direct application of standard duality results in distributionally robust optimization. Third, the problem involves a nonstandard optimal transport structure in which one marginal is not fixed, requiring a joint treatment of transport and distributional optimization.

Our main contributions are summarized as follows.

\begin{enumerate}
    \item \textbf{Modeling contribution.} We extend the robust utility maximization framework for intractable claims in two directions. 
    First, we incorporate $\varphi$-divergence ambiguity at the marginal level, allowing the distribution of the intractable claim to vary within a prescribed ambiguity set. 
    Second, we allow for a general (possibly non-additive) bivariate utility function, extending the classical additive specification used in the existing literature. 
    Together, these features lead to a unified formulation that captures both marginal uncertainty and dependence uncertainty, and can be viewed as a natural extension of existing models such as \cite{li2023robust}.

    \item \textbf{Theoretical contributions.} We establish the existence of optimal solutions in an infinite-dimensional setting and derive uniqueness results under suitable conditions. We further develop a Legendre-Fenchel type duality framework, building on tools from convex analysis (see, e.g., \cite{rockafellar2015convex}, \cite{luenberger1997optimization}, \cite{hiriart2004fundamentals}), that connects the constrained formulation with a penalized counterpart. A key feature of the analysis is that the divergence constraint acts on the marginal distribution, which leads to a nonstandard dual structure. Finally, we propose a feasible numerical method based on unbalanced optimal transport (see, e.g., \cite{chizat2018scaling}) and projected gradient techniques.
\end{enumerate}

To further clarify the relation between our model and existing approaches, we briefly position our work relative to the literature. 

Compared with the robust utility maximization models for intractable claims (see, e.g., \cite{hou2016robust}, \cite{li2023robust}, \cite{chen2026alpha}), our formulation allows for ambiguity in the marginal distribution and a general class of (possibly non-additive) bivariate utility function, which prevents the use of quantile-based methods.

Compared with distributionally robust optimization (see, e.g., \cite{delage2010distributionally}, \cite{goh2010distributionally},  \cite{ben2013robust}, \cite{rahimian2022frameworks}, \cite{jiang2024data}),  where the decision is typically deterministic and the ambiguity is imposed on the distribution of the underlying uncertainty, our model characterizes the decision as a random variable, which is inherited from the expected utility framework, and imposes ambiguity only on the marginal, leading to a different structure. 

Finally, while our formulation is related to optimal transport through the coupling between the decision and the claim, it differs from classical transport problems (see, e.g., \cite{villani2008optimal}) in that one marginal is not fixed.

\subsection{Organization}

The remainder of the paper is organized as follows. Section~\ref{sec:prob_formulation} introduces the model setup and the main assumptions. Section~\ref{sec_existence} proves the existence of an optimal decision via a stability argument rooted in optimal transport theory. Section~\ref{Sec_DualityTheorem} develops a Legendre–Fenchel type duality framework, and establishes the connection between the constrained and penalized RDRO formulations, which serves as a key tool for the subsequent analysis, in particular for establishing uniqueness and enabling tractable reformulations for the numerical approach. Section~\ref{Sec_uniqueness_and_innerexistence} builds on this duality framework, further investigates the structural properties of optimal solutions and proves the uniqueness of the optimal decision.
Section~\ref{Sec_numerical} presents numerical methods and experiments. Finally, Section~\ref{Sec_conclusion} concludes the paper.

\section{{\bf Problem formulation}}\label{sec:prob_formulation}
We start from a financial market formulation to introduce the underlying robust utility maximization problem and its economic interpretation. We then lift this formulation to the general RDRO framework, which enables a systematic analysis and provides the foundation for the subsequent theoretical developments.

\subsection{Robust utility maximization problem under distributional ambiguity}\label{subsec:DRO_invest}

Consider a complete, non-atomic probability space $(\Omega, \mathcal{F}, \mathbb{P})$ and an arbitrage-free market. Assume that the market is complete and has a unique pricing kernel $\rho$, which is $\mathcal{F}$-measurable and satisfies $\mathbb{P}(\rho > 0) = 1$ and $\mathbb{E}[\rho] < \infty$.

Consider a small investor who has purchased an option  which offers a random income $Y=\xi$ at the terminal time (known as an intractable contingent claim, see, e.g., \cite{li2023robust}), she has an estimate of the distribution of the payoff $Y$, and would like to choose a robust investment strategy to maximize the expectation of her utility under the worst-case scenario.

For any initial capital $x > 0$, we may define the set of all attainable terminal wealth via the budget constraint:
\begin{equation}
\mc{A} := \left\{ X \ge 0 : \EE[\rho X] \le x \right\}.
\end{equation}

Then, we formulate the following robust utility maximization problem:
\begin{equation}\label{Primalprob}
 \sup_{X \in \mc{A}} \inf_{Y \sim \nu \in D_{\eta}} \E \left[ U(X,Y)\right],
\end{equation}
where $U(\cdot,\cdot)$ is the utility function of terminal wealth X with intractable claim Y, and Y is an exogenous random variable on $(\Omega, \mc{F},\mb{P})$ with its unknown distribution $\nu = \mP_Y =\mP \circ Y^{-1}$ taking value in $D_{\eta}$, and  $D_{\eta}$ is the ambiguity set with error $\eta$ given by Definition~\ref{Def:Deta}. Generally, we denote the value sets of variables $X$ and $Y$ as $\mc{X} := [0, \infty)$ and $\mc{Y} \subset \mR$, respectively. 

We assume that $Y$ is essentially bounded, i.e.,
\begin{equation}
\|Y\|_{\infty} := \operatorname{ess\,sup} |Y| < \infty.
\end{equation}
It then follows that $\mc{Y} \subset \mR$ is bounded, and hence a compact convex set.

The utility function $U(\cdot)$ satisfies the following assumptions:
\begin{assumption}\label{Asp_U1}
The utility function $U: \mc{X} \times \mc{Y} \to \mathbb{R}$ is continuous. Moreover, for each fixed $y \in \mc{Y}$, the function $U(\cdot,y)$ is bounded, and strictly concave in $x \in \mc{X}$.
\end{assumption}

In this paper, we characterize ambiguity through a $\varphi$-divergence–based uncertainty set (See Section~\ref{sec_preliminaries} for details).
\begin{definition}\label{Def:Deta}
For a benchmark distribution $\nu_0$, let
\begin{equation}\label{Deta}
   D_{\eta}= \{\nu \in \mathcal{P}(\mc{Y}) : D_{\varphi}(\nu, \nu_0)\le \eta  \}.
\end{equation}
\end{definition}
Here, $\mathcal{P}(\mc{Y})$ denotes the set of all probability measures on $\mc{Y}$,  $D_{\varphi}(\nu, \nu_0)$ is the $\varphi$-divergence between $\nu$ and the nominal distribution $\nu_0 \in \mathcal{P}(\mathcal{Y})$, and $\eta \geq 0$ represents the ambiguity tolerance.

Meanwhile, for a general utility satisfying Assumption~\ref{Asp_U1}, we assume the $\varphi$-divergence is generated by a superlinear $\varphi$, such as the Kullback–Leibler divergence or the modified $\chi^2$ distance. For utilities with compact sublevel sets, both superlinear and non-superlinear $\varphi$-divergences can be used, for example, the Hellinger distance or total variation. For additional examples of $\varphi$-divergences, see \cite{bayraksan2015data}.

\subsection{Examples and applications}
We now present several examples of utility functions that satisfy Assumption~\ref{Asp_U1}. They illustrate different forms of interaction between the terminal wealth $X$ and the intractable claim $Y$, including additive, state-dependent, and bounded specifications.

\begin{example}[Additive CARA utility (benchmark case; cf. \cite{li2023robust})]
The utility function
\[
U(x,y) = - \exp\big( -\alpha(x+y)\big), \quad \alpha>0.
\]

This is a CARA utility applied to total wealth $x+y$. This form is consistent with the setting in \cite{li2023robust}. In our framework, it serves as a baseline case.
\end{example}

\begin{example}[State-dependent CARA utility]
The utility function
\[
U(x,y) =  - \exp\big(-a(y)\, x\big), \quad a(y)>0.
\]

Here, $a(y)$ is a positive continuous function that represents state-dependent risk aversion. This allows preferences to vary with $y$, which may describe market or economic conditions. The interaction between $x$ and $y$ is no longer purely additive.
\end{example}

\begin{example}[Saturating linear utility]
The utility function
\[
U(x,y) = \frac{x}{1+x}\, b(y), \quad b(y)>0.
\]

The function $x/(1+x)$ has decreasing marginal utility and is bounded. This captures saturation effects. The factor $b(y)$ is a positive continuous function that scales utility across states.
\end{example}

\begin{example}[Modified power utility]
The utility function
\[
U(x,y) = \frac{x^\alpha}{1+x^\alpha}\, b(y), \quad 0<\alpha<1, \quad b(y)>0.
\]

This is a bounded version of the power utility. It behaves like $x^\alpha$ for small $x$ and remains bounded for large $x$. The factor $b(y)$ is a positive continuous function that scales utility across states.
\end{example}

\begin{remark}
Assumption~\ref{Asp_U1} holds for the above examples provided that $a(y)$ and $b(y)$ are continuous.
\end{remark}

\subsection{The random distributionally robust optimization framework}

In this paper, we are primarily concerned with the existence and uniqueness of optimal solutions to the robust utility maximization problem \eqref{Primalprob}.

Meanwhile, we observe that problem \eqref{Primalprob} is closely related to classical distributionally robust optimization (DRO) formulations (see, e.g., \cite{delage2010distributionally}, \cite{goh2010distributionally},  \cite{ben2013robust}, \cite{rahimian2022frameworks},   \cite{gao2024wasserstein}, \cite{jiang2024data}). A key distinction, however, is that in standard DRO problems the decision variable is deterministic, whereas in our setting the decision variable $X$ is modeled as a random variable. 

This structural difference naturally motivates us to lift the problem to a more general random distributionally robust optimization (RDRO) framework.
Such an abstraction not only captures the essential structure of \eqref{Primalprob}, but also allows for a unified treatment of existence and uniqueness, and extends naturally to a broader class of max–min problems.

Following the above discussion, we introduce the RDRO problem in the constrained form
\begin{equation}\label{Primalprob_cons}
 J^{c}(\eta) : = \sup_{X \in \mc{A}} \inf_{Y \sim \nu \in D_{\eta}} \E \left[ U(X,Y)\right].
\end{equation}
In this formulation, the inner infimum is taken over distributions $\nu$ restricted to the ambiguity set $D_{\eta}$ defined in \eqref{Deta}. The superscript ``$c$'' emphasizes that the problem is posed as a constrained max–min problem. The corresponding penalized formulation and distributional reformulations will be introduced in subsequent sections, where they play a key role in the analysis.

Here and throughout the remainder of the paper, we reinterpret the \textbf{terminal wealth} $X$ and the \textbf{intractable claim} $Y$ as the \textbf{decision variable} $X$ and the \textbf{environment variable} $Y$, respectively. Both are modeled as random variables defined on a complete, non-atomic probability space $(\Omega, \mathcal{F}, \mP)$, taking values in convex, closed sets $\mathcal{X} \subset \mathbb{R}^n$ and $\mathcal{Y} \subset \mathbb{R}^m$, respectively. The decision variable $X$ is constrained to lie in a feasible set $\mc{A}$ (referred to as the decision set), while the distribution $\nu = \mP_Y =\mP \circ Y^{-1}$ of $Y$ is assumed to belong to an ambiguity set $D_{\eta}$,  where $D_{\eta}$ is the ambiguity set with error $\eta$ defined in \eqref{Deta}. The function $U: \mathcal{X} \times \mathcal{Y} \to \mathbb{R}$ represents a bivariate utility function that captures the interaction between the decision and the realized uncertainty.

The analysis of the RDRO problem relies on a set of structural assumptions. These assumptions can be viewed as mild relaxations of those imposed in the financial model introduced in Section~\ref{subsec:DRO_invest}. For clarity, we defer their precise statements to the points where they are needed in the subsequent analysis, so as not to interrupt the flow of the presentation. A detailed verification of these assumptions in the financial setting is provided in Appendix~\ref{app:verification}.

\begin{remark}
For ease of reference, we summarize here the locations of the structural assumptions. 
Assumptions on the choice of the utility function and the $\varphi$-divergence are given in 
Assumptions~\ref{Asp_utility}, \ref{Asp_utility_concave}, \ref{Asp_coercive}, and \ref{Asp_Utility_strict_concave}. 
Assumptions on the decision set are given in 
Assumptions~\ref{Asp_mu_weak_convergent} and \ref{Asp_Subdiff_of_sup}.
\end{remark}

\subsection{Notations and concepts}\label{sec_preliminaries}
This section collects the notations, definitions, and key theoretical concepts used throughout the paper. These preliminaries form the basis for the proofs of the main results. While some of the material is standard, we include it to fix notation and clarify the conventions adopted in our analysis. Readers familiar with these topics may skip this section on a first reading and return to it as needed.


Let $\mathcal{X} \subset \mathbb{R}^n$ be a closed convex set, $\mc{B}(\mc{X})$ the Borel $\sigma$-field on $\mc{X}$. Denote by $\mathcal{P}(\mathcal{X})$ the set of all probability measures on $\mathcal{X}$, and by $\mathcal{M}^+(\mathcal{X})$ the set of all finite positive Radon measures on $\mathcal{X}$.

\subsubsection{Entropy function and $\varphi$-divergence}
We begin by defining the class of entropy functions that characterize the divergence measures used in our RDRO framework.

\begin{definition}[Entropy Function]\label{Def_entropy_func}
A function $\varphi: [0, +\infty) \to [0, +\infty]$ is called an entropy function if it satisfies the following assumptions:
\begin{enumerate}
\item $\varphi$ is convex and lower semi-continuous;
\item The effective domain intersects the open interval $(0, +\infty)$, i.e.,
$$  D(\varphi) := \{s \ge 0 : \varphi(s) < \infty\} \cap (0, +\infty) \neq \emptyset.$$
\end{enumerate}
The growth rate of $\varphi$ at infinity is characterized by
 $\varphi'_{\infty} := \lim_{x\longrightarrow +\infty} \frac{\varphi(s)}{s}$. If $\varphi'_{\infty} = \infty$, then $\varphi$ grows faster than any linear function and is referred to as superlinear.
\end{definition}

Based on this, we define the $\varphi$-divergence between two measures, which generalizes classical notions such as the Kullback–Leibler divergence.
\begin{definition}[$\varphi$-divergence]\label{Def_KLdiv}
Let $\mu, \nu \in \mathcal{M}^+(\mathcal{X})$, and let $\mu = \sigma_{\mu} \nu + \mu^{\perp}$ be the Lebesgue decomposition of $\mu$ with respect to $\nu$. The $\varphi$-divergence $D_{\varphi}: \mathcal{M}^+(\mathcal{X}) \times \mathcal{M}^+(\mathcal{X}) \to [0, \infty]$ is defined as
\begin{equation}\label{phi_div}
D_{\varphi}(\mu, \nu) = \int_{\mc{X}} \varphi \circ \sigma_{\mu} \, \dd \nu + \varphi'_{\infty} \mu^{\perp}(\mc{X})
\end{equation}
with the convention $0 \cdot \infty = 0$.
\end{definition}

The mapping $D_{\varphi}$ is jointly convex, lower semi-continuous, and non-negative; see, e.g., \cite[Corollary 2.9]{liero2018optimal}. Moreover, for any $t > 0$, it satisfies the scaling property $D_{t\varphi} = t D_{\varphi}$. As a special case, if $\varphi = \iota_{\{1\}}$, the indicator function of $\{1\}$, then $D_{\varphi}(\mu, \nu)$ reduces to the indicator divergence:
\begin{equation}\label{Def_indicator_entropy}
\iota\_{{=}}(\mu, \nu) = \begin{cases}
0, & \text{if } \mu = \nu, \\
+\infty, & \text{otherwise},
\end{cases}
\end{equation}
where $\iota_{\mathcal{S}}(x) = 0$ if $x \in \mathcal{S}$, and $+\infty$ otherwise.


\subsubsection{Optimal transport problem and optimal entropy-transport problem}
In this subsection, we first recall the classical OT problem, which seeks a coupling minimizing the expected transport cost between given marginals. Let $c : \mathcal{X} \times \mathcal{Y} \to \mathbb{R}$ be a cost function, and let $\mu \in \mathcal{P}(\mathcal{X})$, $\nu \in \mathcal{P}(\mathcal{Y})$. The classical OT problem seeks a coupling $\gamma$ with marginals $\mu$ and $\nu$ that minimizes the total transport cost:
\begin{equation}\label{OT_1}
\min \left\{  \int_{\mc{X} \times \mc{Y}} c \, \dd \gamma \, : \, \gamma \in \Pi(\mu, \nu) \right\},
\end{equation}
where $\Pi(\mu, \nu)$ denotes the set of positive finite Radon measures on $\mathcal{X} \times \mathcal{Y}$ with marginals $\mu$ and $\nu$. An optimal solution $\gamma$ to \eqref{OT_1} is referred to as an optimal transport plan.

To generalize and allow for penalized marginal deviations, we then consider the optimal entropy-transport problem.  Given entropy functions $F_1$ and $F_2$, and marginals $\mu, \nu \in \mathcal{M}^+(\mathcal{X})$ and $\mathcal{M}^+(\mathcal{Y})$, respectively, the problem is
\begin{equation}
    \inf_{\gamma \in \mathcal{M}^+(\mc{X}\times\mc{Y})} \left\{  \int c(x,y)\dd \gamma(x,y)  + D_{F_1}(\pi^{\sharp}_1 \gamma, \mu)+ D_{F_2}(\pi^{\sharp}_2 \gamma, \nu)\right\},
\end{equation}
where the marginals of $\gamma$ are respectively  defined as $ \pi_1^\sharp \gamma(A) := \int_{\mathcal{Y}} \gamma(A \times \dd y)$,  $\pi_2^\sharp \gamma(B) := \int_{\mathcal{X}} \gamma(\dd x \times B)
$ for all measurable sets $A \subset \mathcal{X}$, $B \subset \mathcal{Y}$.



\subsubsection{Concavity}
As the final part of this section, we introduce definitions and properties of concave functions, which play a central role in the subsequent analysis.
\begin{definition}
Let $C \subset \mathbb{R}^n$ be a convex set. For a convex function $f: C \to \mathbb{R}$, the {\bf subdifferential} of $f$ at $x \in C$ is defined as
$$
\partial f(x) := \left\{ x^* \in \mathbb{R}^n : \langle x^*, y - x \rangle \le f(y) - f(x), \ \forall y \in C \right\}.
$$

For a concave function $g$, the subdifferential is defined by $\partial g(x) := -\partial(-g)(x)$. That is, $x^* \in \partial g(x)$ if and only if
\begin{equation*}
\langle x^*, y - x \rangle \ge g(y) - g(x) , \quad \forall y \in C.
\end{equation*}
\end{definition}

The following properties of concave functions, directly derived from classical convex analysis (see, e.g., \cite{rockafellar2015convex}, \cite{phelps2009convex}), will play a crucial role in our work:

\begin{proposition}[Properties of concave function]\label{Prop_concave_properties}Let $f: \mathbb{R}^n \to \mathbb{R}$ be concave. Then
    \begin{enumerate}
    \item The pointwise infimum of an arbitrary family of concave functions is concave.
    \item $f$ is continuous relative to any relatively open convex subset of its effective domain. Here, continuity relative to a set means that the restriction of $f$ to that set is continuous. 
    \item If $f$ is continuous at $x_0 \in \mathbb{R}^n$, then $\partial f(x_0) \neq \emptyset$.
    \end{enumerate}
\end{proposition}

\vskip 10pt

\section{\bf Existence of the optimal decision of the RDRO problem}\label{sec_existence}
In RDRO problems, the primary objective is to identify an optimal decision $X^*$, while the specific optimal value of the environment variable $Y^*$ is not the main focus. Consequently, this section focuses on establishing the existence of an optimal decision $X^*$.

To facilitate this, we first reformulate the original RDRO problem $J^c(\eta)$ \eqref{Primalprob_cons} into an equivalent distributional formulation $J^{c,d}(\eta)$ \eqref{Primalprob_cons_distribution_f1}. This reformulation enables us to leverage powerful results from OT theory. In particular, the inner infimum in $J^{c,d}(\eta)$ can be interpreted as an OT problem, enabling a more tractable and structured approach for analysis.

To proceed, we first consider the distributional counterpart of the decision set:
\begin{equation}
    \mc{A}_{d} := \{ \mu\in \mc{P}(\mc{X}) : \mu \text{ is the distribution of $X$ for some $X \in \mc{A}$} \}.
\end{equation}

Using this notation, the distributional RDRO problem is defined by
\begin{equation}\label{Primalprob_cons_distribution_f1}
      J^{c,d}(\eta) := \sup_{\mu \in \mc{A}_{d} } \inf_{\substack{ \gamma \in \Pi(\mu,\nu) \\ \nu \in D_{\eta}}}   \int U(x,y)\dd \gamma(x,y).
\end{equation}

The equivalence between the original RDRO problem and its distributional form is summarized as:
\begin{equation}
      J^{c}(\eta) = \sup_{X \in \mc{A}} \inf_{Y \sim \nu \in D_{\eta}} \E \left[ U(X,Y)\right] 
      = \sup_{\mu \in \mc{A}_{d} } \inf_{\substack{ \gamma \in \Pi(\mu,\nu) \\ \nu \in D_{\eta}}} \ \int U(x,y)\dd \gamma(x,y)
      = J^{c,d}(\eta).
\end{equation}

Here, since $\mc{X}$ and $\mc{Y}$ are Polish spaces and $\Omega$ is non-atomic, by the disintegration theorem, for any $\mu \in \mc{A}_{d}$ and corresponding $X \sim \mu$, for any $\gamma \in \Pi(\mu,\nu)$, we can find a random variable $Y$ on $(\Omega,\mc{F},\mP)$ that $(X,Y) \sim \gamma$. (See Appendix~\ref{app:equivalence_of_rv_and_measure})

To prove existence, we impose the following assumptions on the utility function and the decision set.

\begin{assumption}[Utility function] \label{Asp_utility} 
The utility function $U: \mc{X} \times \mc{Y} \to \mathbb{R}$ is continuous and bounded.
\end{assumption}

\begin{assumption}[Decision set]\label{Asp_mu_weak_convergent}
    The decision set $\mc{A}$ satisfies:
\begin{enumerate}
\item $\mc{A}$ is convex;
\item The corresponding distributional decision set $\mc{A}_{d}$ is weakly compact; that is, for any sequence $\{\mu_n\} \subset \mc{A}_{d}$, there exists a subsequence $\{\mu_{n_k}\}$ and $\mu \in \mc{A}_{d}$ such that $\mu_{n_k} \xrightarrow{w} \mu$ (weak convergence).
\end{enumerate}
\end{assumption}

Then, based on the equivalence between the RDRO problem $J^c(\eta)$ and its distributional counterpart $J^{c,d}(\eta)$, we obtain the existence of the optimal decision.

\begin{theorem}\label{THM_existence_OTapproach}
     Suppose that Assumptions~\ref{Asp_utility} and \ref{Asp_mu_weak_convergent} hold. Then, the RDRO problem~\eqref{Primalprob_cons} admits at least one optimal decision $X^* \in \mc{A}$.
\end{theorem}
\begin{proof}
   Consider an optimizing sequence $\{ \mu_n \}_{n\ge 1}$ for the outer supremum in the distributional RDRO problem. By Assumption~\ref{Asp_mu_weak_convergent}, there is a subsequence which is still denoted as $\{ \mu_n \}_{n\ge 1}$, and weakly convergent to a $\mu_0 \in \mc{A}_{d}$. 

   
    Then, for any $\nu_0 \in D_{\eta}$, by \cite[Theorem 5.20]{villani2008optimal}, there exists a subsequence of the optimal transport plans $\{ \gamma_n \}_{n\ge 1}$ that converges weakly to some $\gamma_0 \in \Pi(\mu_0, \nu_0)$, where each $\gamma_n$ is an optimal transport plan between $\mu_n$ and $\nu_n := \nu_0$. We continue to denote this subsequence by ${ \gamma_n }$. Thus,
    \begin{equation*}
        \begin{aligned}
             \inf_{ \gamma \in \Pi(\mu_0,\nu_0) }  \int U(x,y)\dd \gamma(x,y) 
            &=   \int U(x,y)\dd \gamma_0(x,y) 
             =   \lim_{n}   \int U(x,y)\dd \gamma_n(x,y)\\
             & =  \lim_{n} \inf_{ \gamma \in \Pi(\mu_n,\nu_0) }  \int U(x,y)\dd \gamma(x,y) 
             \ge  \lim_{n} \inf_{\substack{ \gamma \in \Pi(\mu_n,\nu)\\ \nu \in D_{\eta} } }  \int U(x,y)\dd \gamma(x,y) \\
            & =  \sup_{\mu\in \mc{A}_{d}} \inf_{\substack{ \gamma \in \Pi(\mu,\nu)\\ \nu \in D_{\eta} } }   \int U(x,y)\dd \gamma(x,y) .
        \end{aligned}
    \end{equation*}
  Taking the infimum over $\nu_0 \in D_{\eta}$ on the left-hand side yields
    \begin{equation}\label{temp_414}
        \inf_{\substack{ \gamma \in \Pi(\mu_0,\nu_0)\\ \nu_0 \in D_{\eta} } }  \int U(x,y)\dd \gamma(x,y) 
             \ge  \sup_{\mu\in \mc{A}_{d}} \inf_{\substack{ \gamma \in \Pi(\mu,\nu)\\ \nu \in D_{\eta} } }  \int U(x,y)\dd \gamma(x,y) .
    \end{equation}
    Thus, $\mu_0$ attains the supremum, and by definition of $\mc{A}_{d}$, there exists $X^* \in \mc{A}$ whose distribution is $\mu_0$. Hence, the RDRO problem admits an optimal decision. Therefore, $\mu_0 \in \mc{A}_{d}$ attains the supremum on the right-hand side of \eqref{temp_414}, and thus serves as an optimal solution to the outer optimization in the distributional RDRO problem.
     Then by the definition of $\mc{A}_{d}$, the original RDRO problem~\eqref{Primalprob_cons} admits at least one optimal decision $X^* \in \mc{A}$.
\end{proof}

\begin{remark}
    It is worth emphasizing that this existence proof does not depend on any particular structure or properties of the ambiguity set $D_{\eta}$. Therefore, the result naturally extends to other choice of the ambiguity set.
\end{remark}

We have thus established the existence of an optimal decision for the RDRO problem~\eqref{Primalprob_cons}, under the utility assumption~\ref{Asp_utility} and the decision assumption~\ref{Asp_mu_weak_convergent}. To proceed further and address the uniqueness of the optimal decision, we require a suitable duality result. This duality result is not only central for proving the uniqueness of the optimal decision but also plays a critical role in designing efficient numerical algorithms, as demonstrated in later sections.

 
\vskip 15pt
\section{{\bf Duality}} \label{Sec_DualityTheorem}

In this section, we establish a duality between the distributional RDRO problem and its penalized counterpart. 
This duality result enables the derivation of key properties of the value function, such as convexity and monotonicity, and lays the theoretical groundwork for proving the uniqueness of the optimal decision in Section~\ref{Sec_uniqueness_and_innerexistence} and for developing the numerical methods presented in Section~\ref{Sec_numerical}.

We begin by rewriting the inner admissible set in the distributional RDRO problem~\eqref{Primalprob_cons_distribution_f1} through the following transformation:
\begin{equation*}
    \begin{aligned}
         \{ \gamma \in \Pi(\mu,\nu):  \nu \in D_{\eta}\} 
        \!=\!  \{ \gamma \in \Pi(\mu,\nu):  D_{\varphi}(\nu, \nu_0)\le \eta \}\! = \! \{ \gamma \in \mc{P}(\mc{X}\times\mc{Y}):  \pi^{\sharp}_1 \gamma=\mu , \  D_{\varphi}(\pi^{\sharp}_2 \gamma, \nu_0)\le \eta \},     
    \end{aligned}
\end{equation*}
which characterizes the feasible set for the inner infimum in \eqref{Primalprob_cons_distribution_f1} in terms of the joint distribution $\gamma$. Under this formulation, the distributional RDRO problem~\eqref{Primalprob_cons_distribution_f1} can be equivalently expressed as
\begin{equation}\label{Primalprob_cons_distribution}
     J^{c,d}(\eta) = \sup_{\mu \in \mc{A}_{d} } \inf_{ \gamma \in  \Gamma(\mu,\eta)}   \int U(x,y)\dd \gamma(x,y),
\end{equation}
where the admissible set $\Gamma(\mu,\eta)$ is defined as
\begin{equation} \label{gamma_admissible_set}
        \Gamma(\mu,\eta)  = \{ \gamma \in \mathcal{P}(\mc{X}\times\mc{Y}): \pi^{\sharp}_1 \gamma = \mu,\  D_{\varphi}(\pi^{\sharp}_2 \gamma, \nu_0)\le \eta \}.
\end{equation}


Based on the reformulation in~\eqref{Primalprob_cons_distribution}, we now derive the corresponding dual problem via Lagrangian duality. This leads to a penalized formulation in which the divergence constraint is incorporated into the objective as a penalty term. The dual problem takes the form:
\begin{equation}\label{Primalprob_penalty_distribution_formal}
 J^{p,d}(\theta) := \sup_{\mu \in \mc{A}_{d}} \inf_{\gamma \in \Gamma_{\mu}} \left\{  \int U(x,y)\dd \gamma(x,y)  +  \theta \mc{R}_2(\gamma)\right\}, 
\end{equation}
where $\mc{R}_2(\gamma):= D_{\varphi}(\pi^{\sharp}_2 \gamma, \nu_0)$
and $\Gamma_{\mu} : = \{ \gamma \in \mathcal{P}(\mc{X}\times\mc{Y}): \pi^{\sharp}_1 \gamma = \mu\}$.

To establish the duality between the distributional RDRO problem $J^{c,d}(\eta)$ and its penalized counterpart $J^{p,d}(\theta)$, we introduce a set of technical assumptions in addition to Assumptions~\ref{Asp_utility} and~\ref{Asp_mu_weak_convergent}. These additional assumptions are essential for ensuring the well-posedness of the duality framework and for applying standard results from convex analysis.

\begin{assumption}[Concavity] \label{Asp_utility_concave}
The utility function $U(x,y)$ is concave with respect to $x \in \mathcal{X}$ for any fixed $y \in \mathcal{Y}$.
\end{assumption}
 
\begin{assumption}[Coercive]\label{Asp_coercive}
One of the following two coercivity conditions holds:
\begin{enumerate}
\item The function $\varphi$ is superlinear;
\item The utility function $U$ has compact sublevel sets, that is, for every $\alpha \in \mathbb{R}$, the following set $  L_{\alpha} $ is compact, 
\[
L_{\alpha} := \{ (x, y) \in \mathcal{X} \times \mathcal{Y} : U(x, y) \le \alpha \}.
\]
\end{enumerate}
\end{assumption}

\begin{assumption}[Subdifferential]\label{Asp_Subdiff_of_sup}
For any $\mu \in \mc{A}_{d}$, the set $\Gamma_{\mu}$ is weakly compact, where $\Gamma_{\mu} : = \{ \gamma \in \mathcal{P}(\mc{X}\times\mc{Y}): \pi^{\sharp}_1 \gamma = \mu\}$.
\end{assumption}


\begin{remark}
    Here, the concavity of the utility function in  is not necessary if we merely prove a duality between the RDRO problem and its penalized counterpart. However, without the concavity, some duality results will be weakened, we will explain this more precisely in Remark~\ref{rmk:duality_without_concavity}.

    And the coercive assumption~\ref{Asp_coercive} is correponding to the coercivity in \cite{liero2018optimal}, which enables us to apply results from the entropy optimal transport theory.
\end{remark}

In the following, we prove the duality theorem between the distributional RDRO problem and its penalized counterpart in three steps. First, we establish the connection between the distributional penalized RDRO formulation and the original penalized RDRO problem, which enables the analysis of the concavity of the value function $J^{p,d}(\theta)$. Next, by exploiting this concavity and applying results from \cite[Theorem 4.4.2]{hiriart2004fundamentals}, we characterize the subdifferential $\partial J^{p,d}(\theta)$. Finally, this subdifferential characterization is used to complete the proof of the duality theorem.

\subsection{Concavity of the value functions}
To establish the concavity of the value function $J^{p,d}(\theta)$, we consider the penalized RDRO problem parameterized by $\theta \geq 0$, defined by
\begin{equation}\label{Primalprob_penalty_formal}
     J^{p}(\theta) := \sup_{X \in \mc{A}} \inf_{Y \sim \nu\in \mathcal{P}(\mc{Y})}\left\{ \E [ U(X,Y)] + \theta D_{\varphi}(\nu, \nu_0)\right\}.
\end{equation}

The penalized RDRO problem \eqref{Primalprob_penalty_formal} is equivalent to its distributional formulation \eqref{Primalprob_penalty_distribution_formal}, as captured by the following equality:
\begin{equation*}
     J^{p}(\theta) 
     = \sup_{\mu \in \mc{A}_{d} } \inf_{\substack{\gamma \in \mathcal{P}(\mc{X}\times\mc{Y})\\ \pi^{\sharp}_1 \gamma = \mu}} \left\{ \int U(x,y)\dd \gamma(x,y) + \theta D_{\varphi}(\pi^{\sharp}_2 \gamma, \nu_0)\right \}
       =J^{p,d}(\theta).
\end{equation*}

Then, the functions $J^{p}(\theta)$ and $J^{p,d}(\theta)$ inherit the following concavity properties.

\begin{lemma}\label{Lemma_Jp_concave}
Suppose that Assumptions~\ref{Asp_utility}, \ref{Asp_mu_weak_convergent}, and \ref{Asp_utility_concave} hold. Then the functions \(J^{p}(\theta)\) and \(J^{p,d}(\theta)\) are concave on the domain \([0, +\infty)\).
\end{lemma}
\begin{proof} 
   We begin by establishing the concavity of  $J^p(\theta)$.  Under Assumption~\ref{Asp_utility}, the utility function  $U(x,y)$ is concave in $x \in \mathcal{X}$. Consequently, for any fixed $Y \sim \nu$, the mapping
$$
(\theta, X) \mapsto \mathbb{E}[U(X, Y)] + \theta\, D_{\varphi}(\nu, \nu_0)
$$
is jointly concave on $[0, +\infty) \times \mc{A}$. As the infimum of a family of concave functions is itself concave (see, e.g., \cite{rockafellar2015convex}), it follows that the function
    \begin{equation}\label{Jp_theta_X}
        J^p(\theta,X) := \inf\limits_{Y \sim \nu \in \mathcal{P}(\mc{Y})} \mathbb{E}\left\{ \left[ U(X,Y) \right] + \theta D_{\varphi}\right\}
    \end{equation}
    is also jointly concave in $(\theta, X)$.

    For any $\theta_1, \theta_2 \geq 0$, let $X_{\theta_1}$ and $X_{\theta_2}$ be arbitrary decisions in $\mc{A}$. By  Assumption~\ref{Asp_mu_weak_convergent}, $\mc{A}$ is convex,  as such, for any $\lambda \in [0, 1]$, the convex combination $\lambda X_{\theta_1} + (1 - \lambda) X_{\theta_2}$ also belongs to $\mc{A}$.

     Using the joint concavity of $J^p(\theta, X)$ in $(\theta, X)$, we have
    \begin{equation*}
        \begin{aligned}
             \lambda J^p(\theta_1, X_{\theta_1}) + (1 - \lambda) J^p(\theta_2, X_{\theta_2}) 
            \leq & J^p \Big( \lambda \theta_1 + (1 - \lambda) \theta_2, \lambda X_{\theta_1} + (1 - \lambda) X_{\theta_2} \Big) \\
            \leq & \sup_{X \in \mc{A}} J^p \Big( \lambda \theta_1 + (1 - \lambda) \theta_2, X \Big) 
            =  J^p \Big( \lambda \theta_1 + (1 - \lambda) \theta_2 \Big).
        \end{aligned}
    \end{equation*}
    
Taking the supremum over $X_{\theta_1} \in \mc{A}$ and $X_{\theta_2} \in \mc{A}$ yields
    \begin{equation*}
        \begin{aligned}
             \lambda J^p(\theta_1) + (1 - \lambda) J^p(\theta_2) 
           = \lambda\sup_{X_{\theta_1} \in \mc{A}}  J^p(\theta_1, X_{\theta_1}) +  (1 - \lambda) \sup_{X_{\theta_2} \in \mc{A}} J^p(\theta_2, X_{\theta_2})         \le  J^p \Big( \lambda \theta_1 + (1 - \lambda) \theta_2 \Big).
        \end{aligned}
    \end{equation*}
    Therefore, $J^p(\theta)$ is concave. The concavity of $J^{p,d}(\theta)$ follows directly from the identity $J^{p,d}(\theta) = J^{p}(\theta)$.
\end{proof}

\subsection{Characterization of the subdifferential}
  
Based on the concavity of $J^{p,d}(\theta)$, we now proceed to characterize its subdifferential $\pt J^{p,d}(\theta)$, which is formalized in Lemma~\ref{Lemma_opt_R_is_in_subdifferential}. 
 
To enhance clarity in the subsequent analysis, we introduce the following notations in TABLE~1, which will be used throughout the remainder of the paper:
\begin{table}[htbp]\label{TABLE1}
\footnotesize
\centering
\renewcommand{\arraystretch}{1.5} 
\begin{tabular}{|p{0.5\linewidth}|p{0.35\linewidth}|}
\hline
\textbf{Notation} & \textbf{Interpretation} \\
\hline
$J^{p,d}_{o}(\theta,\mu,\gamma) := \int U(x,y)\dd \gamma(x,y) + \theta \mc{R}_2(\gamma)$
& Objective function of $J^{p,d}(\theta)$ \\
\hline
$J^{p,d}_{in}(\theta,\mu):=\inf\limits_{\gamma \in \Gamma_{\mu}}  J^{p,d}_{o}(\theta,\mu,\gamma)$
& Inner infimum of $J^{p,d}(\theta)$ \\
\hline
$S_{\mu}(\theta)$ 
& All optimizers of $J^{p,d}_{in}(\theta,\mu)$ \\
\hline
$S(\theta)$ 
& All optimizers of $J^{p,d}(\theta)$ \\
\hline
\end{tabular}
\caption{Definitions of various notations.}
\end{table}

\begin{lemma}\label{Lemma_opt_R_is_in_subdifferential}
    Suppose that Assumptions \ref{Asp_utility}, \ref{Asp_mu_weak_convergent}, \ref{Asp_utility_concave} - \ref{Asp_Subdiff_of_sup} hold.
    Then, for any $\theta \in (0,+\infty)$, $\partial J^{p,d}(\theta)$ is nonempty, and
    \begin{equation}
        \partial J^{p,d}(\theta) \subset \{\mathcal{R}_2(\gamma_{\theta}^*) : (\mu_{\theta}^*, \gamma_{\theta}^*) \in S(\theta) \}.
    \end{equation}
\end{lemma}
\begin{proof} The proof is organized in four steps as follows: 

\textbf{Step 1}: \textit{Non-emptiness of the subdifferential $\partial J^{p,d}(\theta)$ for $\theta \in (0,+\infty) $.}

By Lemma~\ref{Lemma_Jp_concave}, the function $J^{p,d}(\theta)$ is concave on $[0,\infty)$. 
It then follows from Proposition~\ref{Prop_concave_properties} that the subdifferential $\partial J^{p,d}(\theta)$ is nonempty for all $\theta \in (0,\infty)$.

\textbf{Step 2}: \textit{Concavity of $J^{p,d}(\theta, \mu)$.}

Consider the inner infimum of $J^{p,d}(\theta)$:
\begin{equation*}
      J^{p,d}_{in}(\theta,\mu) :=\inf_{\gamma \in \Gamma_{\mu}}  J^{p,d}_{o}(\theta,\mu,\gamma)= \inf_{\gamma \in \Gamma_{\mu}} \left\{  \int U(x,y)\dd \gamma(x,y)  + \theta \mc{R}_2(\gamma)\right \}.
\end{equation*}

For a fixed $\mu$, $J^{p,d}(\theta, \mu)$ is the infimum of a family of linear functions in $\theta$, and is thus concave on $\theta \in [0, \infty)$. Consequently, the subdifferential $\partial_{\theta} J^{p,d}(\theta, \mu)$ is well-defined for all $\theta \in [0, \infty)$.

\textbf{Step 3}: \textit{Inclusion relation $\pt J^{p,d}(\theta) \subset \pt_{\theta} J^{p,d}(\theta,\mu_{\theta}^*)$ for any outer optimizer $\mu_{\theta}^*$ of $J^{p,d}(\theta)$.}

Based on  the definition of the subdifferential, for any $\theta_0 \in [0, \infty)$ and  $x^* \in \pt J^{p,d}(\theta_0)$, we have 
\begin{equation*}
    \langle x^*, \theta - \theta_0 \rangle \ge J^{p,d}(\theta) -  J^{p,d}(\theta_0), \quad \forall \theta  \in [0, \infty) .
\end{equation*}
Then, for any outer optimizer $\mu_{\theta_0}^*$ of $J^{p,d}(\theta_0)$, it follows that
\begin{equation*}
   \begin{aligned}
        \langle x^*, \theta - \theta_0 \rangle \ge & J^{p,d}(\theta) -  J^{p,d}(\theta_0) 
        =  \sup\limits_{\mu \in \mc{A}_{d}} J^{p,d}_{in}(\theta,\mu) -  J^{p,d}(\theta_0,\mu_{\theta_0}^*) \\
        \ge & J^{p,d}(\theta,\mu_{\theta_0}^*) -  J^{p,d}(\theta_0,\mu_{\theta_0}^*) , \quad \forall \theta \ge 0, 
   \end{aligned}
\end{equation*}
i.e.,  $x^* \in \pt_{\theta} J^{p,d}(\theta_0,\mu_{\theta_0}^*)$.
Thus, for any $\theta \in [0, \infty)$ and any outer optimizer $\mu_{\theta}^*$, we have $$\pt J^{p,d}(\theta) \subset \pt_{\theta} J^{p,d}(\theta,\mu_{\theta}^*).$$

\textbf{Step 4}: \textit{Characterizing the subdifferential $\pt_{\theta} J^{p,d}_{in}(\theta,\mu)$ for $\theta \in (0,+\infty)$ and $\mu \in \mc{A}_{d}$.}

Fix $\mu \in \mc{A}_{d}$, as the $\varphi$-divergence is lower semi-continuous, $J^{p,d}_{o}(\theta,\mu,\gamma)$ is also lower semi-continuous. By Assumption~\ref{Asp_Subdiff_of_sup}, for any $\mu \in \mc{A}_{d}$, the set $\Gamma_{\mu}$ is weakly compact. Meanwhile, by \cite[Section 6]{villani2008optimal}, there are many distances that could metrize the weak topology on $\mathcal{P}(\mathcal{X} \times \mathcal{Y})$. Thus, applying \cite[Theorem 4.4.2]{hiriart2004fundamentals}, we obtain the following representation of the subdifferential:
 \begin{equation*}
\pt_{\theta} J^{p,d}_{in}(\theta,\mu) = \operatorname{conv} \left\{ \bigcup_{\gamma \in S_{\mu}(\theta)} \pt_{\theta} J^{p,d}_{o}(\theta,\mu,\gamma) \right\} ,
\end{equation*}
where $\operatorname{conv}\{\cdot\}$ denotes the convex hull. Moreover, for each $\gamma$,
\begin{equation*}
    \pt_{\theta} J^{p,d}_{o}(\theta,\mu,\gamma) = \mc{R}_2(\gamma).
\end{equation*}
Hence,
\begin{equation}\label{Temp_2}
    \pt_{\theta} J^{p,d}_{in}(\theta,\mu) = \operatorname{conv} \left\{ \bigcup_{\gamma \in S_{\mu}(\theta)} \pt_{\theta} J^{p,d}_{o}(\theta,\mu,\gamma) \right\}=\operatorname{conv} \left\{ \mc{R}_2(\gamma) : \gamma \in S_{\mu}(\theta) \right\}.
\end{equation}

Next, we show that the set $\left\{ \mathcal{R}_2(\gamma) : \gamma \in S_{\mu}(\theta) \right\}$ on the right-hand side of \eqref{Temp_2} is convex. Note that the inner infimum can be formulated as an optimal entropy-transport problem:
\begin{equation*}
    \begin{aligned}
        J^{p,d}_{in}(\theta,\mu) & = \inf_{\gamma \in \Gamma_{\mu}} \left\{  \int U(x,y)\dd \gamma(x,y)  + \theta \mc{R}_2(\gamma) \right\}\\
        & =  \inf_{\gamma \in \mathcal{M}^+(\mc{X}\times\mc{Y})}   \left\{\int U(x,y)\dd \gamma(x,y)  + D_{\iota_{\{1\}}}(\pi^{\sharp}_1 \gamma, \mu)+ D_{\theta \varphi}(\pi^{\sharp}_2 \gamma, \nu_0)\right\} .
    \end{aligned}
\end{equation*}
By \cite[Theorem 3.3]{liero2018optimal}, the solution set $S_{\mu}(\theta)$ is convex and compact.

Now, take any two elements $x_1^*, x_2^* \in \left\{ \mathcal{R}_2(\gamma) : \gamma \in S_{\mu}(\theta) \right\}$, with corresponding $\gamma_1, \gamma_2 \in S_{\mu}(\theta)$ such that $x_i^* = \mathcal{R}_2(\gamma_i)$ for $i=1,2$. For any $\lambda \in (0,1)$, define the convex combination $\gamma_0 = \lambda \gamma_1 + (1-\lambda) \gamma_2$. As $S_{\mu}(\theta)$ is convex, $\gamma_0 \in S_{\mu}(\theta)$. Moreover, by optimality,
\begin{equation*}
    \int U(x,y)\dd \gamma_i + \theta \mc{R}_2(\gamma_i) =J^{p,d}_{in}(\theta,\mu) , \quad i=0,1,2.
\end{equation*}
Combining with $\gamma_0 = \lambda \gamma_1 + (1-\lambda) \gamma_2$, we obtain $\mc{R}_2(\gamma_0)=\lambda\mc{R}_2(\gamma_1)+(1-\lambda)\mc{R}_2( \gamma_2 )$. Thus, the set $ \left\{ \mc{R}_2(\gamma) : \gamma \in S_{\mu}(\theta) \right\}$ is convex. Consequently, Equation \eqref{Temp_2} simplifies to the equality
    $$\pt_{\theta} J^{p,d}_{in}(\theta,\mu) =  \left\{ \mc{R}_2(\gamma) : \gamma \in S_{\mu}(\theta) \right\}.$$
\vskip 5pt 
\textbf{Conclusion:}\ \  Combining the results from Steps 3 and 4, we conclude that for any $\theta \in (0,\infty)$ and any outer optimizer $\mu = \mu_{\theta}^*$ of $J^{p,d}(\theta)$, the subdifferential satisfies
$$\partial J^{p,d}(\theta) \subset \left\{ \mathcal{R}_2(\gamma) : \gamma \in S_{\mu}(\theta) \right\}.$$
Taking the intersection over all such outer optimizers yields
\begin{equation*}
\partial J^{p,d}(\theta) \subset \bigcap_{\mu_{\theta}^* \text{ is an outer optimizer}} \left\{ \mathcal{R}_2(\gamma) : \gamma \in S_{\mu_{\theta}^*}(\theta) \right\} \subset \left\{ \mathcal{R}_2(\gamma_{\theta}^*) : (\mu_{\theta}^*, \gamma_{\theta}^*) \in S(\theta) \right\}.
\end{equation*}
This completes the proof of the desired inclusion.
\end{proof}
\vskip 5pt
\subsection{Completion of the duality theorem}
We now complete the proof of the dual results stated in Theorem~\ref{THM_eta_theta_duality}, which rigorously establishes the duality relationship between the RDRO problem $J^{c,d}(\eta)$~\eqref{Primalprob_cons_distribution} and its penalized counterpart~\eqref{Primalprob_penalty_distribution_formal}. This proof mainly relies on the concavity property from Lemma~\ref{Lemma_Jp_concave} and the subdifferential characterization provided by Lemma~\ref{Lemma_opt_R_is_in_subdifferential}.

\begin{theorem}[Dual relation]\label{THM_eta_theta_duality}
Suppose that Assumptions \ref{Asp_utility}, \ref{Asp_mu_weak_convergent} and \ref{Asp_utility_concave} - \ref{Asp_Subdiff_of_sup} hold. Then the following statements hold:
\begin{enumerate}
\item  Let $\eta \ge 0$ and suppose that $(\mu_{\eta }^*,\gamma_{\eta }^*)$ is an optimizer of the RDRO problem $J^{c,d}(\eta )$~\eqref{Primalprob_cons_distribution}. Then there exists a scalar $\theta = \theta(\eta,\mu_{\eta }^*)  \ge 0 $ such that  $(\mu_{\eta }^*,\gamma_{\eta }^*)$ also solves the penalized RDRO problem $J^{p,d}(\theta)$~\eqref{Primalprob_penalty_distribution_formal}. Moreover, $\theta = \theta(\eta,\mu_{\eta }^*)$ satisfies the duality relation:
    \begin{equation}\label{dual_equ_eta}
        J^{c,d}(\eta) = \max_{\theta \ge 0}  \{J^{p,d}(\theta) - \theta \eta\} .
    \end{equation}
\item  Conversely, let $\theta \in (0,+\infty)$ and suppose that $J^{p,d}(\theta )$ admits at least one optimizer. Then for any $\eta  \in \pt J^{p,d}(\theta )$, there exists an optimizer $(\mu_{\theta}^*,\gamma_{\theta}^*)$ of $J^{p,d}(\theta)$ that also solves  the RDRO problem  $J^{c,d}(\eta)$~\eqref{Primalprob_cons_distribution}, with $\eta = \mc{R}_2(\gamma_{\theta}^*)$.

In addition, any subgradient $\eta \in \partial J^{p,d}(\theta)$ satisfies the dual relation:
    \begin{equation}\label{dual_equ_theta}
        J^{p,d}(\theta) = \min_{\eta \ge 0} \{J^{c,d}(\eta) + \theta \eta\} .
    \end{equation}
    \end{enumerate}
\end{theorem}
\begin{proof}
(1)  To prove the first part, we construct a Lagrangian formulation based on convexity properties. As the divergence $D_{\varphi}$ is jointly convex and the projection to the second marginal $\pi^{\sharp}_2: \gamma \mapsto \pi^{\sharp}_2 \gamma$ is linear, the composition $\mc{R}_2(\gamma) = D_{\varphi}(\pi^{\sharp}_2 \gamma, \nu_0)$ is convex in $\gamma$. Moreover, the expected objective term $\int U(x,y)\dd \gamma(x,y)$ is linear in $\gamma$. Therefore, the RDRO problem $J^{c,d}(\eta)$~\eqref{Primalprob_cons_distribution}  admits the following Lagrangian reformulation: 
     \begin{equation*}
         \begin{aligned}
             J^{c,d}(\eta) = &   \sup_{\mu \in \mc{A}_{d} } \inf_{ \gamma \in  \Gamma(\mu,\eta)}   \int U(x,y)\dd \gamma(x,y)\\
                        = &  \sup_{\mu \in \mc{A}_{d} } \inf_{ \gamma \in \Gamma_{\mu}} \sup_{\theta \ge 0} \left\{  \int U(x,y)\dd \gamma(x,y) + \theta \left(\mc{R}_2(\gamma)- \eta\right)\right\},
         \end{aligned}
     \end{equation*}
where $\Gamma(\mu,\eta)  = \{ \gamma \in \mathcal{P}(\mc{X}\times\mc{Y}): \pi^{\sharp}_1 \gamma = \mu,\  D_{\varphi}(\pi^{\sharp}_2 \gamma, \nu_0)\le \eta \}$, $\Gamma_\mu := \{ \gamma \in \mathcal{P}(\mathcal{X} \times \mathcal{Y}) : \pi_1^{\sharp} \gamma = \mu \}$ are both convex sets by their definition.

By Assumption~\ref{Asp_Subdiff_of_sup}, $\Gamma_{\mu}$ is weakly compact, and for $\gamma_0$ s.t. $\pi^{\sharp}_2 \gamma_0 = \nu_0$, the slater condition holds, interchanging the order of the supremum and the infimum, we obtain
     \begin{equation*}
         \begin{aligned}
             &  \sup_{\mu \in \mc{A}_{d} } \inf_{ \gamma \in \Gamma_{\mu}} \sup_{\theta \ge 0}  \left\{ \int U(x,y)\dd \gamma(x,y) + \theta \left(\mc{R}_2(\gamma)- \eta\right)\right\}\\
             = &  \sup_{\mu \in \mc{A}_{d} }  \sup_{\theta \ge 0} \inf_{ \gamma \in \Gamma_{\mu}}  \left\{ \int U(x,y)\dd \gamma(x,y) + \theta \left(\mc{R}_2(\gamma)- \eta\right)\right\}\\
             = &   \sup_{\theta \ge 0} \sup_{\mu \in \mc{A}_{d} } \inf_{ \gamma \in \Gamma_{\mu}}  \left\{ \int U(x,y)\dd \gamma(x,y) + \theta \left(\mc{R}_2(\gamma)- \eta\right)\right\}\\
             = &  \sup_{\theta \ge 0} \{ J^{p,d}(\theta)  - \theta \eta  \}.
         \end{aligned}
     \end{equation*}
Combining the above, we have the desired dual representation:
    \begin{equation}\label{Sup_theta_ge_0}
        J^{c,d}(\eta)=   \sup_{\theta \ge 0}\{ J^{p,d}(\theta) - \theta \eta \}.
    \end{equation}

    We now return to the proof of the theorem. For a given $\eta \ge  0$,  suppose that $(\mu^*_{\eta}, \gamma^*_{\eta})$ is an optimal solution to the RDRO problem $J^{c,d}(\eta)$.  First, fix any $\mu \in \mc{A}_{d}$ and consider the inner infimum. By  \cite[Section 8.3, Theorem 1 ]{luenberger1997optimization}, there exists a scalar $\theta = \theta(\eta, \mu) \ge 0$ such that
    \begin{equation}\label{eq:luenberger_inner_dual}
      \inf_{ \gamma \in  \Gamma(\mu,\eta)}   \int U(x,y)\dd \gamma(x,y)
                        =    \inf_{ \gamma \in \Gamma_{\mu}}\left\{  \int U(x,y)\dd \gamma(x,y) + \theta \left(\mc{R}_2(\gamma)- \eta\right) \right\},
    \end{equation}
    and any minimizer $\gamma^*$ on the left-hand side is also a minimizer on the right-hand side and satisfies the complementary slackness condition  $\theta\left (\mc{R}_2(\gamma^*)- \eta\right)=0$.

    In particular, for the optimal solution $(\mu^*_{\eta}, \gamma^*_{\eta})$ to the RDRO problem $J^{c,d}(\eta)$, let $\theta := \theta(\eta, \mu^*_{\eta})$. Then we have
    \begin{equation*}
        \begin{aligned}
            J^{c,d}(\eta) =& \int U(x,y)\dd \gamma^*_{\eta}(x,y)
            =   \inf_{ \gamma \in  \Gamma(\mu^*_{\eta},\eta)}  \int U(x,y)\dd \gamma(x,y)\\
            = &   \inf_{ \gamma \in \Gamma_{\mu^*_{\eta}}} \left\{ \int U(x,y)\dd \gamma(x,y) + \theta \left(\mc{R}_2(\gamma)- \eta \right)\right\} \\
            \le & \sup_{\mu \in \mc{A}_{d} } \inf_{ \gamma \in \Gamma_{\mu}} \left\{ \int U(x,y)\dd \gamma(x,y) + \theta \left(\mc{R}_2(\gamma)- \eta \right)\right\}\\
             = &    J^{p,d}(\theta)  - \theta \eta
             \le \sup_{\theta \ge 0}\{ J^{p,d}(\theta)  - \theta \eta\} 
             =  J^{c,d}(\eta).
        \end{aligned}
    \end{equation*}
As the final and initial expressions coincide, all inequalities must be equalities. We therefore conclude that $(\mu^*_{\eta}, \gamma^*_{\eta})$ also solves the penalized RDRO problem
    \begin{equation*}
       J^{p,d}(\theta) = \sup_{\mu \in \mc{A}_{d} } \inf_{ \gamma \in \Gamma_{\mu}} \left\{  \int U(x,y)\dd \gamma(x,y) + \theta \mc{R}_2(\gamma)\right\}
    \end{equation*}
for $\theta = \theta(\eta,\mu_{\eta }^*)$.
\vskip 3pt
(2) To prove the second part of the theorem, we rely on Lemmas~\ref{Lemma_Jp_concave} and~\ref{Lemma_opt_R_is_in_subdifferential} stated earlier.

We begin by establishing the equality in \eqref{dual_equ_theta}. By Lemma~\ref{Lemma_Jp_concave}, the function $J^{p,d}(\theta)$ is concave on the domain $[0, +\infty)$. Then for any $\theta_0 \in (0, +\infty)$ and any $\eta \in \partial J^{p,d}(\theta_0)$, the supremum
\begin{equation*}
    \sup_{\theta \ge 0} \{J^{p,d}(\theta) - \theta \eta \}
\end{equation*}
is attained at $\theta = \theta_0$. Hence, it follows from \eqref{Sup_theta_ge_0} that
\begin{equation*}
    J^{c,d}(\eta) = \sup_{\theta \ge 0} \{ J^{p,d}(\theta) - \theta \eta\}  = J^{p,d}(\theta_0) - \theta_0 \eta.
\end{equation*}
As the expression \eqref{Sup_theta_ge_0} holds for all $\eta \ge 0$, combining it with the identity above yields that $\eta \in \partial J^{p,d}(\theta_0)$ solves the dual problem:
\begin{equation*}
    J^{p,d}(\theta_0) = \min_{\eta \ge 0} \{ J^{c,d}(\eta) + \theta_0 \eta\} .
\end{equation*}
Now we return to the main argument.  
    For any $\theta \in (0,+\infty)$ and $\eta \in \pt J^{p,d}(\theta)$, Lemma~\ref{Lemma_opt_R_is_in_subdifferential} ensures that there exists at least one optimizer $(\mu_{\theta}^*, \gamma_{\theta}^*)$ of $J^{p,d}(\theta)$ such that $\eta = \mc{R}_2(\gamma_{\theta}^*)$.  Then,
    \begin{equation*}
        \begin{aligned}
             J^{p,d}(\theta) -& \theta \eta 
          =\left[\int U(x,y)\dd \gamma_{\theta}^*(x,y) + \theta \mc{R}_2(\gamma_{\theta}^*) \right]- \theta\eta  \\
          =& \inf_{\gamma \in \Gamma_{\mu_{\theta}^*}} \left[ \int U(x,y)\dd \gamma(x,y) + \theta \left(\mc{R}_2(\gamma)- \eta\right) \right] \\
         \le & \inf_{\gamma \in \Gamma_{\mu_{\theta}^*}}  \sup_{\theta \ge 0} \left[  \int U(x,y)\dd \gamma(x,y) + \theta \left(\mc{R}_2(\gamma)- \eta \right) \right] \\
         = & \inf_{\substack{\gamma \in \Gamma_{\mu_{\theta}^*}\\ \mc{R}_2(\gamma)\le \eta}}   \int U(x,y)\dd \gamma(x,y)
         =  \inf_{\gamma \in \Gamma(\mu_{\theta}^*,\eta)}   \int U(x,y)\dd \gamma(x,y) \\
         \le & \sup_{\mu \in \mc{A}_{d}} \inf_{\gamma \in \Gamma(\mu,\eta)}   \int U(x,y)\dd \gamma(x,y) 
         = J^{c,d}(\eta) 
         =  J^{p,d}(\theta) - \theta \eta.       
        \end{aligned}
    \end{equation*}
As both sides are equal, all inequalities above  become equalities. It then follows that
    \begin{equation*}
        \begin{aligned}
        J^{c,d}(\eta) =& \sup_{\mu \in \mc{A}_{d}} \inf_{\gamma \in \Gamma(\mu,\eta)}  \int U(x,y)\dd \gamma(x,y)
          = \inf_{\gamma \in \Gamma(\mu_{\theta}^*,\eta)}  \int U(x,y)\dd \gamma(x,y) \\
          = & \int U(x,y)\dd \gamma_{\theta}^*(x,y) + \theta \mc{R}_2(\gamma_{\theta}^*)- \theta\eta
          =  \int U(x,y)\dd \gamma_{\theta}^*(x,y).
        \end{aligned}
    \end{equation*}
    Thus, the pair $(\mu_{\theta}^*, \gamma_{\theta}^*)$ also solves the RDRO problem $J^{c,d}(\eta)$~\eqref{Primalprob_cons_distribution}.
\end{proof}

\begin{remark}\label{rmk:duality_without_concavity}
    If Assumption~\ref{Asp_utility_concave}, namely the concavity of the utility function, is removed, the concavity of $J^{p,d}(\theta)$ on the entire domain $[0,+\infty)$ can no longer be established via Lemma~\ref{Lemma_Jp_concave}. 
    However, the first part of the above proof still yields concavity of $J^{p,d}(\theta)$ on a restricted domain (see Appendix~\ref{app:dual_without_concavity}). 

    As a consequence, Lemma~\ref{Lemma_opt_R_is_in_subdifferential} and the duality result in Theorem~\ref{THM_eta_theta_duality} remain valid for $\theta$ in the restricted domain 
    \begin{equation*}
        \{\theta = \theta(\eta, \mu^*_{\eta}) : \eta \ge 0\},
    \end{equation*}
    where the mapping $\theta(\eta, \mu)$ is defined in the first part of the above proof.
\end{remark}

\subsection{Summary and implications}
In this section, we have established the duality result stated in Theorem~\ref{THM_eta_theta_duality}. This theorem provides a valuable framework for understanding the relationship between the parameter $\eta$ in the constrained RDRO problem $J^{c,d}(\eta)$ and the parameter $\theta$ in the penalized RDRO problem $J^{p,d}(\theta)$, highlighting their connection through a shared optimal solution. The result has key implications, notably for ensuring the uniqueness of the optimal decision and informing the design of numerical algorithms.

First, Equations~\eqref{dual_equ_eta} and~\eqref{dual_equ_theta} imply that $J^{c,d}(\eta)$ is decreasing and convex in $\eta$, while $J^{p,d}(\theta)$ is increasing and concave in $\theta$. These analytical properties are further supported by the numerical experiments presented in Section~\ref{Sec_numerical}.

Second, in Section~\ref{Sec_uniqueness_and_innerexistence}, we use the dual relation to obtain the existence of an optimizer for the constrained RDRO problem by establishing the existence of an optimizer for the penalized problem. This approach is particularly advantageous, as the penalized problem is typically more tractable from both theoretical and numerical perspectives.

Finally, the parameter correspondence provided by the duality theorem significantly enhances the implementation of numerical methods. In particular, it enables the constrained RDRO problem to be efficiently solved via its penalized counterpart, which often exhibits superior convergence properties and is more amenable to algorithmic optimization.

\begin{remark}
The dual relationship between the parameters $\eta$ and $\theta$ is established primarily at the level of an existence result. More precisely, for any given $\theta \in (0,\infty)$, one can associate a corresponding ambiguity radius $\eta = \mc{R}_2(\gamma_{\theta}^*)$,
where $\gamma_{\theta}^*$ denotes an inner optimizer of the penalized RDRO problem $J^{p,d}(\theta)$.

In contrast, for a prescribed $\eta \ge 0$, the corresponding dual parameter $\theta$ generally does not admit an explicit representation.

This phenomenon is consistent with classical results in convex duality, particularly those related to Legendre-Fenchel duality; see, e.g., \cite{rockafellar2015convex}, \cite{luenberger1997optimization}, and \cite{hiriart2004fundamentals}. In such frameworks, the correspondence between primal and dual parameters is typically characterized via subdifferential relations of value functions. While the mapping $\theta \mapsto \eta$ can often be constructed through optimal solutions of the penalized problem, the inverse mapping is generally only implicitly determined through optimality conditions (e.g., subgradient inclusions), and thus lacks a closed-form expression.
\end{remark}

\vskip 10pt
\section{\bf Uniqueness of the optimal decision} \label{Sec_uniqueness_and_innerexistence}

Building on the duality result under the concavity assumption, we now show that a mild strengthening of this assumption guarantees the uniqueness of the optimal decision. 

Although the existence of an optimal environment variable $Y^*$ is not essential for solving the decision problem itself, it plays a crucial role in establishing the uniqueness of the optimal decision. 
Accordingly, we first present a corresponding existence result for the optimal environment variable as a preparatory step.


\subsection{Existence of the optimizer}
Before proving the uniqueness of the optimal decision, we first establish the existence of an optimal environment variable, which is instrumental for the subsequent uniqueness result.

\begin{theorem}\label{Thm_existence_inner}
Suppose that the assumptions in Theorem~\ref{THM_eta_theta_duality} hold. Then the following statements hold: 
\begin{enumerate}
    \item For any $\theta \ge 0$, the penalized RDRO problem $J^{p,d}(\theta)$~\eqref{Primalprob_penalty_distribution_formal} admits at least one optimizer $(\mu_{\theta}^*, \gamma_{\theta}^*)$. 
    \item For any $\eta \in \bigcup_{\theta \in (0,+\infty)} \pt J^{p,d}(\theta)$, the distributional RDRO problem $J^{c,d}(\eta)$ \eqref{Primalprob_cons_distribution_f1} admits at least one optimizer $(\mu_{\eta}^*, \gamma_{\eta}^*)$. Subsequently, the original RDRO problem $J^{c}(\eta)$~\eqref{Primalprob_cons}  also admits at least one optimizer $(X_{\eta}^*, Y_{\eta}^*)$.
\end{enumerate}
\end{theorem}
\begin{proof}
    Following essentially the same arguments as  Theorem~\ref{THM_existence_OTapproach}, by extracting a maximizing sequence from the outer supremum in the penalized RDRO problem $J^{p,d}(\theta)$ and applying Theorem 5.20 in \cite{villani2008optimal}, we conclude that for any $\theta \ge 0$,  there exists at least one outer optimizer $\mu_{\theta}^* \in \mc{A}_{d}$.  
    
For the inner infimum
    \begin{equation*}
        \begin{aligned}
            J^{p,d}_{in}(\theta,\mu) & = \inf_{\gamma \in \Gamma_{\mu}}   \left\{\int U(x,y)\dd \gamma(x,y)  + \theta \mc{R}_2(\gamma)\right\} \\
            & =  \inf_{\gamma \in \mathcal{M}^+(\mc{X}\times\mc{Y})}   \left\{\int U(x,y)\dd \gamma(x,y)  + D_{\iota_{\{1\}}}(\pi^{\sharp}_1 \gamma, \mu)+ D_{\theta \varphi}(\pi^{\sharp}_2 \gamma, \nu_0) \right\}
        \end{aligned}
    \end{equation*}
    of $J^{p,d}(\theta)$, which can be interpreted as an optimal entropy-transport problem, Theorem 3.3 in \cite{liero2018optimal} guarantees the existence of at least one optimizer  $\gamma_{\theta}^*$ corresponding to $\mu = \mu_{\theta}^*$. 
    
    Hence, for any $\theta \ge 0$, the penalized RDRO problem $J^{p,d}(\theta)$ admits at least one optimizer.  
    
    Furthermore, for any $\eta \in \bigcup_{\theta \in (0,+\infty)} \pt J^{p,d}(\theta)$, the existence of an optimizer $(\mu_{\eta}^*, \gamma_{\eta}^*)$ for the distributional RDRO problem $J^{c,d}(\eta)$ follows directly from Theorem~\ref{THM_eta_theta_duality}. By applying the disintegration theorem along with the definition of the distributional decision set $\mc{A}_{d}$, we can construct a corresponding random vector $(X_{\eta}^*, Y_{\eta}^*)$ that serves as an optimizer for the original RDRO problem $J^{c}(\eta)$. 
\end{proof}

\subsection{Uniqueness of the optimal decision}\label{subsec_uniqueness}

With the existence of an optimal environment variable established in the previous subsection, we now turn to the uniqueness of the optimal decision. 
Under the strengthened assumption that the utility function $U(x, y)$ is strictly concave in the decision variable $x \in \mc{X}$ (see Assumption~\ref{Asp_Utility_strict_concave}), we can guarantee that the optimal decision $X^*$ is unique.


\begin{assumption}[Strict Concavity]\label{Asp_Utility_strict_concave}
The utility function $U(x, y)$ is strictly concave to $x \in \mc{X}$ for each fixed $y \in \mc{Y}$.
\end{assumption}

When the utility function is strictly concave in $x \in \mc{X}$, based on the existence of the inner optimizer, we can derive the uniqueness of the optimal decision $X^*$.

\begin{theorem}[Uniqueness of the optimal decision]\label{THM_unique}
Suppose that the assumptions of Theorem~\ref{THM_eta_theta_duality} hold, except that Assumption~\ref{Asp_utility_concave} is replaced by the stronger strict concavity condition in Assumption~\ref{Asp_Utility_strict_concave}.  
 Then, for any $\eta \in \bigcup_{\theta \in (0,+\infty)} \pt J^{p,d}(\theta)$, the RDRO problem~\eqref{Primalprob_cons} admits at most one optimal decision $X \in \mc{A}$.
\end{theorem}

\begin{proof}
Assume, for contradiction, that there exist two distinct optimal decisions $X_1, X_2 \in \mc{A}$ solving the RDRO problem~\eqref{Primalprob_cons}.

We first claim that for any optimal decision $X^*$ of the RDRO problem~\eqref{Primalprob_cons}, there exists a corresponding inner optimizer $Y^*$.

To see this, consider the reduced problem with the decision set $\mc{A}$ restricted to the singleton $\tilde{\mc{A}} := \{X^*\}$.  As all assumptions of Theorem~\ref{Thm_existence_inner} still hold on this restricted set, there exists at least one  optimizer $Y^*$ such that the pair $(X^*, Y^*)$ solves
\begin{equation*}
    \tilde{J}^{c}(\eta) := \sup_{X \in \tilde{\mc{A}}} \inf_{Y \sim \nu \in D_{\eta}} \E \left[ U(X,Y) \right] = \inf_{Y \sim \nu \in D_{\eta}} \E \left[ U(X^*, Y) \right],
\end{equation*}
which establishes the existence of an inner optimizer $Y^*$ associated with any optimal decision $X^*$.

Applying this to $X_1$ and $X_2$, we obtain corresponding inner optimizers $Y_1$ and $Y_2$ satisfying
\begin{equation*}
\begin{aligned}
\mathbb{E}\left[ U(X_1, Y_1) \right] &= \inf_{Y \sim \nu \in D_{\eta}} \mathbb{E}\left[ U(X_1, Y) \right] = \sup_{X \in \mc{A}} \inf_{Y \sim \nu \in D_{\eta}} \mathbb{E}\left[ U(X, Y) \right], \\
\mathbb{E}\left[ U(X_2, Y_2) \right] &= \inf_{Y \sim \nu \in D_{\eta}} \mathbb{E}\left[ U(X_2, Y) \right] = \sup_{X \in \mc{A}} \inf_{Y \sim \nu \in D_{\eta}} \mathbb{E}\left[ U(X, Y) \right].
\end{aligned}
\end{equation*}
Because  $\mc{A}$ is convex, for any $\lambda \in (0,1)$, the convex combination $X_0 := \lambda X_1 + (1 - \lambda) X_2\in\mc{A}$. Let $Y_0$ be an inner optimizer  corresponding to $X_0$, i.e.,
\begin{equation*}
\mathbb{E}\left[ U(X_0, Y_0) \right] = \inf_{Y \sim \nu \in D_{\eta}} \mathbb{E}\left[ U(X_0, Y) \right].
\end{equation*}
By Assumption~\ref{Asp_Utility_strict_concave}, the function $U(x, y)$ is strictly concave in $x$. Therefore,
\begin{equation}
    \begin{aligned}
      \inf_{Y \sim \nu \in D_{\eta}} \mathbb{E}&\left[ U(X_0,Y) \right]
      =  \mathbb{E}\left[ U(X_0,Y_0) \right] 
      =  \mathbb{E}\left[ U(\lambda X_1 +(1-\lambda)X_2,Y_0) \right]\\
      > & \lambda \mathbb{E}\left[ U(X_1,Y_0)\right] + (1-\lambda) \mathbb{E}\left[ U(X_2,Y_0)\right]\\
      \ge & \lambda \inf_{Y \sim \nu \in D_{\eta}} \mathbb{E}\left[ U(X_1,Y)\right] + (1-\lambda) \inf_{Y \sim \nu \in D_{\eta}} \mathbb{E}\left[ U(X_2,Y)\right]\\
      = & \lambda \sup_{X \in \mc{A}} \inf_{Y \sim \nu \in D_{\eta}} \mathbb{E}\left[ U(X,Y) \right]+ (1-\lambda) \sup_{X \in \mc{A}} \inf_{Y \sim \nu \in D_{\eta}} \mathbb{E}\left[ U(X,Y)\right] \\
      = & \sup_{X \in \mc{A}} \inf_{Y \sim \nu \in D_{\eta}} \mathbb{E}\left[ U(X,Y)\right].
    \end{aligned}
\end{equation}
This strict inequality contradicts the optimality of $X_1$ and $X_2$.  Thus the optimal decision of the RDRO problem~\eqref{Primalprob_cons} must be unique.

\end{proof}
\begin{remark}
    The same uniqueness argument applies to the penalized RDRO problem~\eqref{Primalprob_penalty_formal} with minor modifications.  Specifically, it suffices to replace $\mathbb{E}\left[ U(X, Y) \right]$ by $\mathbb{E}\left[ U(X, Y) \right] + \theta D_{\varphi}(\nu, \nu_0)$ and relax the constraint $\inf\limits_{Y \sim \nu \in D_{\eta}}$ to $\inf\limits_{Y \sim \nu \in \mathcal{P}(\mc{Y})}$. The rest of the proof remains valid. Therefore, strict concavity of 
 $U(x, y)$ in $x$ also ensures uniqueness of the outer optimizer for the penalized problem.

\end{remark}


\vskip 15pt
\section{\bf Numerical results}\label{Sec_numerical}

This section presents a numerical study of the RDRO problem. The primary objective is to approximate the implicit relationship between key parameters and to illustrate a feasible computational approach. 

To this end, we work with the penalized formulation, which facilitates numerical implementation. Based on this formulation, we develop a two-step iterative procedure to approximate the optimizer, and recover a numerical solution to the original problem~\eqref{Primalprob_cons} with the numerical result that illustrates relationship between key parameters.

\subsection{The penalized RDRO problem under discrete measures}

We begin by discretizing the penalized RDRO problem in~\eqref{Primalprob_penalty_distribution_formal}. Assume that the sample space $\Omega$ is partitioned into $n$ disjoint events $\Omega = \bigcup_{i=1}^n \Omega_i$ based on the information $\mathcal{F}$ generated by the decision variable $X$. Let $\mathbb{P}(\Omega_i) = p_i$, and define the corresponding probability vector $\mbf{p}^\top = (p_1, p_2, \ldots, p_n)$. Let $\mbf{x}^\top = (x_1, x_2, \ldots, x_n)$ denote the values of $X$ over these partitions, so that the constraint $X \in \mc{A}$ becomes $\mbf{x} \in \mc{A}$ in the discretized setting.

Similarly, under the information $\tilde{\mathcal{F}}$ related to the environmental variable $Y$, suppose $\Omega$ is partitioned into $r$ disjoint events $\Omega = \bigcup_{j=1}^r \tilde{\Omega}_j$. Let $\bm{\nu}_0^\top = (\nu_1^0, \nu_2^0, \ldots, \nu_r^0)$ and $\bm{\nu}^\top = (\nu_1, \nu_2, \ldots, \nu_r)$ denote the nominal and true distributions of $Y$, respectively. Let $\mbf{y}^\top := (y_1, y_2, \ldots, y_r)$ be the corresponding realizations of $Y$ on these partitions.

\begin{definition}
Under this discretization, the penalized RDRO problem is formulated as follows:
\begin{equation}\label{Discreteprob_num}
 D(\theta;\mbf{p},\bm{\nu}_0)  = \sup_{\mbf{x} \in \mc{A}} 
\inf_{\bm{\gamma} \in \mR^{n\times r}_{+}}   f(\mbf{x},\bm{\gamma};\theta,\mbf{p},\bm{\nu}_0) ,
\end{equation}
where the objective function is given by
\begin{equation}
   f(\mbf{x},\bm{\gamma};\theta,\mbf{p},\bm{\nu}_0) = \sum_{i=1}^n \sum_{j=1}^r \gamma_{i,j} U(x_i, y_j)+  D_{F_1}(\pi_1^{\#}\bm{\gamma}, \bm{p})  + \theta D_{\varphi}(\pi_2^{\#}\bm{\gamma}, \bm{\nu}_0), 
\end{equation} 
and the indicator-type divergence term is defined by
\begin{equation}
   D_{F_1}(\pi_1^{\#}\bm{\gamma}, \bm{p}) = \begin{cases}
           0 , \text{if} \  \  \pi_1^{\#}\bm{\gamma}=\mbf{p},\\
           + \infty, \text{otherwise}.
       \end{cases}
\end{equation}
The value of the inner infimum is denoted by\begin{equation}\label{NUM_inner_inf}
   f(\mbf{x} ;\theta,\mbf{p},\bm{\nu}_0) = \inf_{\bm{\gamma} \in \mR^{n\times r}_{+}}   f(\mbf{x},\bm{\gamma};\theta,\mbf{p},\bm{\nu}_0).
\end{equation}
\end{definition}
For simplicity, we assume that the utility function $U(\mbf{x}, \mbf{y})$, the objective function $f(\mbf{x}, \bm{\gamma}; \theta, \mbf{p}, \bm{\nu}_0)$, and the inner infimum function $f(\mbf{x}; \theta, \mbf{p}, \bm{\nu}_0)$ are all continuously differentiable in $\mbf{x}$. In cases where this smoothness condition does not hold, the projected gradient method introduced in Algorithm~\ref{Algo_IPGMSA} can be replaced by a projected subgradient method; see, e.g., \cite{nedic2001incremental}.


\subsection{Two-step algorithm}
To numerically solve the discretized penalized RDRO problem~\eqref{Discreteprob_num}, we propose a two-step iterative algorithm. This approach separates the problem into two components: a scaling algorithm for the inner entropy transport problem and a projected gradient method (PGM) for the outer optimization.

\vspace{1em}
\noindent\textbf{Step 1: Scaling algorithm for the inner entropy-regularized problem.}
We begin by recalling the scaling algorithm from \cite[Algorithm 1]{chizat2018scaling}, adapted to our setting. For clarity, we restate it below.

\begin{algorithm}[htbp]
\caption{Scaling Algorithm  (Sinkhorn-Type Iteration)}\label{Algo_SA}
\begin{algorithmic}
\STATE \textbf{Input:} the operators $\operatorname{proxdiv}_{F_1}$, $\operatorname{proxdiv}_{F_2}$; cost matrix $C := U(x_i, y_j)$; regularization $\epsilon$; probability vectors $\mbf{p}$, $\bm{\nu}_0$
\STATE Initialize $K = (K_{ij}) \leftarrow \exp(-C_{ij} / \epsilon)$, $\mbf{b} \leftarrow \mathbf{1}_r$
\REPEAT
  \STATE $\mbf{a} \leftarrow \operatorname{proxdiv}_{F_1}(K( \mbf{b} \odot \bm{\nu}_0 ),\epsilon)$
  \STATE $\mbf{b} \leftarrow \operatorname{proxdiv}_{F_2}(K^{\top}(\mbf{a} \odot \mbf{p} ),\epsilon)$
\UNTIL{stopping criterion is met, or a maximum number l of iterations is reached}
\STATE Compute $\bm{\gamma} = (\gamma_{ij}) \leftarrow (a_i K_{ij} b_j)$
\STATE \textbf{Output:} Transport plan $\bm{\gamma}$
\end{algorithmic}
\end{algorithm}



Here, $\odot$ denotes elementwise multiplication and $\mathbf{1}_r$ is the all-ones vector of length $r$. The operators $\operatorname{proxdiv}_{F_1}$ and $\operatorname{proxdiv}_{F_2}$ are derived from the chosen divergence function. For example, when $D_\varphi$ is the Kullback–Leibler (KL) divergence, these take the form: $$\operatorname{proxdiv_{F_1}}(\mbf{s},\epsilon) = \mbf{p} \oslash \mbf{s}=\left( \frac{p_i}{s_i} \right)_{i=1,...,n}, \operatorname{proxdiv_{F_2}}(\mbf{s},\epsilon) =\left(  \left( \frac{\nu_j^0}{s_j} \right)^{\frac{\lambda}{\lambda + \epsilon}}e^{\frac{-\epsilon}{\lambda + \epsilon}} \right)_{j=1,...,r}.$$

According to \cite{chizat2018scaling}, the scaling algorithm converges to the solution of a regularized version of the inner minimization problem~\eqref{NUM_inner_inf} with regularization parameter $\epsilon > 0$:
\begin{equation}\label{NUM_inner_inf_regularized}
    f(\mbf{x}; \theta, \mbf{p}, \bm{\nu}_0, \epsilon) := \inf_{\bm{\gamma} \in \mathbb{R}_+^{n \times r}} f(\mbf{x}, \bm{\gamma}; \theta, \mbf{p}, \bm{\nu}_0, \epsilon) = \inf_{\bm{\gamma} \in \mathbb{R}_+^{n \times r}} \{f(\mbf{x}, \bm{\gamma}; \theta, \mbf{p}, \bm{\nu}_0) + \epsilon \mathcal{H}(\bm{\gamma})\},
\end{equation}
where $\mathcal{H}(\bm{\gamma})$ denotes an entropy term.

The convergence of the scaling algorithm is summarized as follows.
\!\!\! \!\!\!\!\!\!
\begin{theorem}\label{Thm_convergence}
(1) As $\epsilon \to 0$, the regularized value $f(\mbf{x}; \theta, \mbf{p}, \bm{\nu}_0, \epsilon)$ converges to the unregularized value  $f(\mbf{x}; \theta, \mbf{p}, \bm{\nu}_0)$.

(2) Algorithm~\ref{Algo_SA} converges to the unregularized value  $\bm{\gamma}^*_\epsilon$ of the regularized problem~\eqref{NUM_inner_inf_regularized}, with convergence rate $\operatorname{KL}(\bm{\gamma}^{(l)}_{\epsilon},\bm{\gamma}^{*}_{\epsilon}) \le C_{\epsilon}/l$, where $\bm{\gamma}^{(l)}_\epsilon$ denotes the output after $l$ iterations, $\operatorname{KL}(\cdot, \cdot)$ is the Kullback–Leibler divergence, and $C_\epsilon$ is a constant independent of $\mbf{x}$.
\end{theorem}
\begin{proof}
    See \cite[Theorem 4.1]{chizat2018scaling}. 
\end{proof}

\vspace{1em}
\noindent\textbf{Step 2: PGM for the outer optimization.}

For the outer supremum, by \cite[Theorem 4.4.2]{hiriart2004fundamentals}, the subdifferential of $f$ is given by
\[
\partial_{\mbf{x}} f(\mbf{x} ;\theta,\mbf{p},\bm{\nu}_0) = \operatorname{conv} \left\{ \bigcup\limits_{\bm{\gamma} \text{ solves } f(\mbf{x} ;\theta,\mbf{p},\bm{\nu}_0)} \partial_{\mbf{x}} f(\mbf{x}, \bm{\gamma} ;\theta,\mbf{p},\bm{\nu}_0) \right\}.
\]
Under the assumption that the functions are continuously differentiable with respect to $\mbf{x}$, the subdifferential reduces to the gradient
\[
\nabla_{\mbf{x}} f(\mbf{x} ;\theta,\mbf{p},\bm{\nu}_0) = \nabla_{\mbf{x}} f(\mbf{x}, \bm{\gamma}^* ;\theta,\mbf{p},\bm{\nu}_0),
\]
for any optimizer $\bm{\gamma}^*$ of the inner problem~\eqref{NUM_inner_inf}. 


Then, we propose the two-step algorithm as follows.

\begin{algorithm}[htbp]
\caption{IPGM-SA (Inexact Projected Gradient Method + Scaling Algorithm)} \label{Algo_IPGMSA}
\begin{algorithmic}
\STATE \textbf{Input:} initial point $\mbf{x}_0$, vectors $\mbf{y}$, $\mbf{p}$, $\bm{\nu}_0$, step size $\alpha$, utility function $U$, partial derivative $U_x$
\STATE Initialize $\mbf{x} \gets \mbf{x}_0$
\REPEAT
    \STATE Compute cost matrix $C \gets U(\mbf{x}, \mbf{y})$
    \STATE Compute transport plan $\bm{\gamma} \gets \text{ScalingAlgo}(\mbf{p}, \bm{\nu}_0, C, \epsilon)$
    \STATE Compute gradient: $\nabla f_{\bm{\gamma}}(\mbf{x}) \gets \left( \sum_{j=1}^r \gamma_{ij} U_x(x_i, y_j) \right)_{1 \le i \le n}$
    \STATE Update: $\mbf{x} \gets \mc{P}_{\mc{A}}(\mbf{x} - \alpha \nabla f_{\bm{\gamma}}(\mbf{x}))$
\UNTIL{stopping criterion is met}
\STATE \textbf{Output:} optimal decision $\mbf{x}$, transport plan $\bm{\gamma}$
\end{algorithmic}
\end{algorithm}

Here, $\mc{P}_{\mc{A}}$ denotes the projection onto the closed convex set $\mc{A}$, and the gradient $\nabla f_{\bm{\gamma}}(\mbf{x}) := \nabla_{\mbf{x}} f(\mbf{x}, \bm{\gamma}; \theta, \mbf{p}, \bm{\nu}_0, \epsilon)$ is evaluated at the current transport plan $\bm{\gamma}$.


We next establish the convergence of Algorithm~\ref{Algo_IPGMSA} to an approximate optimizer of the discretized RDRO problem.

\begin{theorem} 
    Let ${\mbf{x}_k}$ be the sequence generated by Algorithm~\ref{Algo_IPGMSA}, and let $\mbf{x}^*$ be the optimal solution to Problem~\eqref{Discreteprob_num}. Then, for any error tolerance $\delta > 0$, there exist regularization parameter $\epsilon > 0$ and inner iteration number $l \in \mathbb{N}$ such that: 
    \begin{equation} \label{num_converge}
        \| \mbf{x}_k - \mbf{x}^* \| \le (1 - r)^k + C_0 \delta
    \end{equation} 
    for some constants $r \in (0,1)$ and $C_0 > 0$.
\end{theorem}

\begin{proof}
Let $\bm{\gamma}_k^*$ and $\bm{\gamma}_{k,\epsilon}^*$ denote the optimal solutions to Problems~\eqref{NUM_inner_inf} and~\eqref{NUM_inner_inf_regularized} with $\mbf{x} = \mbf{x}_k$, respectively. Let $\bm{\gamma}_{k,\epsilon}^{(l)}$ be the output of Algorithm~\ref{Algo_SA} after $l$ iterations, for regularization parameter $\epsilon$ and input $\mbf{x}_k$.

By the convergence result in Theorem~\ref{Thm_convergence} and Pinsker's inequality, for any $\delta > 0$, there exist $\epsilon > 0$ and $l \in \mathbb{N}$ such that
\(
\| \bm{\gamma}_{k,\epsilon}^{(l)} - \bm{\gamma}_{k,\epsilon}^* \|_1 + \| \bm{\gamma}_{k,\epsilon}^* - \bm{\gamma}_k^* \|_1 \le \delta 
\)
uniformly for all $\mbf{x}_k \in \mc{A}$.

Let $\nabla f_{k,\bm{\gamma}}$ denote the gradient computed using $\bm{\gamma}_{k,\epsilon}^{(l)}$, and let $\nabla f := \nabla_{\mbf{x}} f(\mbf{x}; \theta, \mbf{p}, \bm{\nu}_0)$ be the exact gradient. Define $\mbf{z}_k := \mbf{x}_k - \alpha \nabla f(\mbf{x}_k)$. Then
\begin{align*}
   & \| \mbf{x}_{k+1} - \mc{P}_{\mc{A}}(\mbf{z}_k) \|_1 
    = \left\| \mc{P}_{\mc{A}}(\mbf{x}_k - \alpha \nabla f_{k,\bm{\gamma}}(\mbf{x}_k)) - \mc{P}_{\mc{A}}(\mbf{x}_k - \alpha \nabla f(\mbf{x}_k)) \right\|_1 \\
    \le & \left\| \alpha \left( \nabla f_{k,\bm{\gamma}}(\mbf{x}_k) - \nabla f(\mbf{x}_k) \right) \right\|_1 
    = \alpha \sum_{i=1}^n \left| \sum_{j=1}^r \left( \gamma_{k,\epsilon,ij}^{(l)} - \gamma_{k,ij}^* \right) U_x(x_{k,i}, y_j) \right| \\
    \le & \alpha C_U \| \bm{\gamma}_{k,\epsilon}^{(l)} - \bm{\gamma}_k^* \|_1 \le \alpha C_U \delta,
\end{align*}
where $C_U > 0$ is a constant depending only on the partial derivative $U_x$. The result then follows from the convergence theory of inexact projected gradient methods; see, e.g.,  \cite{patrascu2018convergence}.
\end{proof}

\subsection{Numerical examples}
In this subsection, we present a numerical experiment to illustrate the performance of the proposed method based on the robust utility maximization problem introduced in Section~\ref{subsec:DRO_invest}.

We consider the discrete penalized RDRO problem~\eqref{Discreteprob_num} under the following settings. The decision variable $X$ is discretized into $n = 50$ outcomes, and the environmental variable $Y$ takes $r = 2$ values with payoffs $(y_1, y_2) = (0, 1)$. Both the nominal distribution $\mbf{p} \in \mathbb{R}^n$ and the reference distribution $\bm{\nu}_0 \in \mathbb{R}^r$ are assumed to be uniform. The ambiguity parameter is set to $\theta = 1$, and the $\varphi$-divergence is specified as the KL divergence. The utility function is chosen to be CARA utility with risk aversion parameter $\alpha = 0.5$.

The decision set is defined as 
\[
\mc{A} = \left\{ \mbf{x} \in \mathbb{R}^n : x_i \ge 0 \text{ for all } i, \; \mathbb{E}[\mbf{x}^\top \mbf{m}] \le x_0 \right\},
\]
where $x_0 = 1$, and $\mbf{m}$ represents the pricing kernel. The vector $\mbf{m}$ is generated via importance sampling using exponential weights derived from a Girsanov transformation. The numerical results are summarized in Figures~\ref{fig:relation_eta_theta} and~\ref{fig:solution}.

\begin{figure}[htbp]
  \centering
  \begin{subfigure}[t]{0.30\textwidth}
    \includegraphics[width=\textwidth]{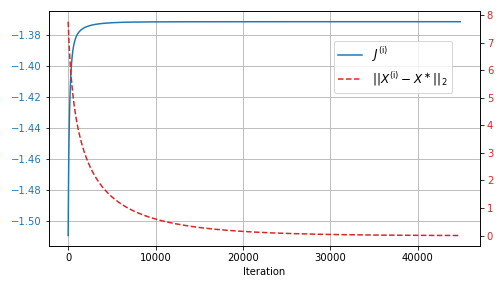}
    \caption{Convergence behavior of Algorithm~\ref{Algo_IPGMSA}}
    \label{fig:convergence}
  \end{subfigure}
  \hfill
  \begin{subfigure}[t]{0.26\textwidth}
    \includegraphics[width=\textwidth]{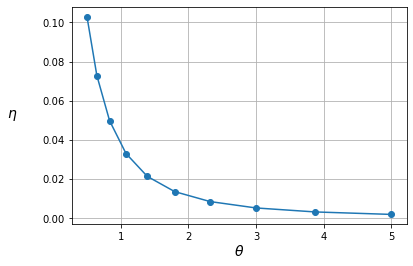}
    \caption{$\eta$ vs. $\theta$}
    \label{fig:eta_theta}
  \end{subfigure}
  \hfill
  \begin{subfigure}[t]{0.27\textwidth}
    \includegraphics[width=\textwidth]{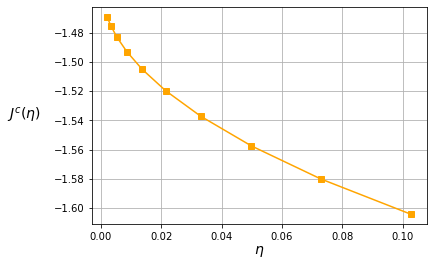}
    \caption{$J^c(\eta)$ vs. $\eta$}
    \label{fig:eta_value_c}
  \end{subfigure}
  \caption{Convergence behavior and parameter relationships.}
  \label{fig:relation_eta_theta}
\end{figure}

Figure~\ref{fig:convergence} illustrates the convergence behavior of Algorithm~\ref{Algo_IPGMSA}, evaluated using the objective value $J^{(i)} := f(\mbf{x}^{(i)}, \bm{\gamma}^{(i)}; \theta, \mbf{p}, \bm{\nu}_0)$ and the residual norm $\nm{\mbf{x}^{(i)}-\mbf{x}^{*}}$, where $(\mbf{x}^{(i)}, \bm{\gamma}^{(i)})$ denote the iterates at step $i$, and $\mbf{x}^*$ is the final numerical solution. The plots confirm the convergence of both the decision variable and the value function.

Figure~\ref{fig:eta_theta} illustrates the correspondence between the parameters $\eta$ and $\theta$ in the RDRO and penalized RDRO problems, respectively. According to Theorem~\ref{THM_eta_theta_duality}, they satisfy the relation $\eta = \mc{R}_2(\bm{\gamma}_{\theta}^*)$. As the penalty parameter $\theta$ increases, the penalized formulation enforces a stronger penalty on the divergence term, corresponding to a smaller tolerance $\eta$ in the original constraint-based formulation. This correspondence, which is only implicitly established in the theoretical analysis, provides a principled way to select $\theta$ from a given $\eta$, thereby complementing the theoretical results and guiding numerical implementation.

Figure~\ref{fig:eta_value_c} shows the dependence of the optimal value function $J^c(\eta)$ on the ambiguity tolerance $\eta$. As $\eta$ increases, representing greater allowable model misspecification, the decision-maker adopts a more conservative strategy, resulting in a lower expected utility. Moreover, the rate of decline in $J^c(\eta)$ diminishes with increasing $\eta$, suggesting a decreasing marginal impact of further enlarging the ambiguity set.

\begin{figure}[htbp]
  \centering
  \begin{subfigure}[t]{0.45\textwidth}
    \includegraphics[width=\textwidth]{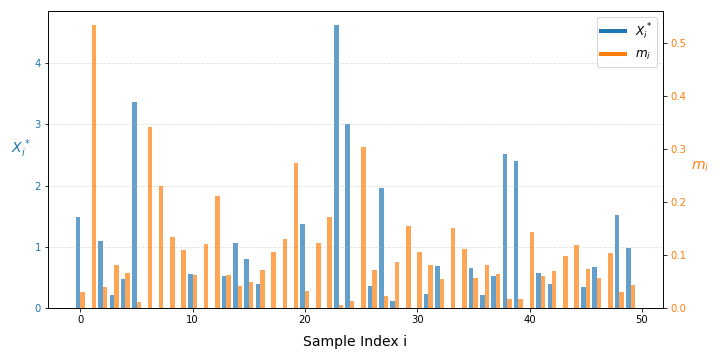}
    \caption{Optimal decision $\mbf{x}^*$ corresponding to pricing kernel $\mbf{m}$}
    \label{fig:X_opt}
  \end{subfigure}
  \hfill
  \begin{subfigure}[t]{0.36\textwidth}
    \includegraphics[width=\textwidth]{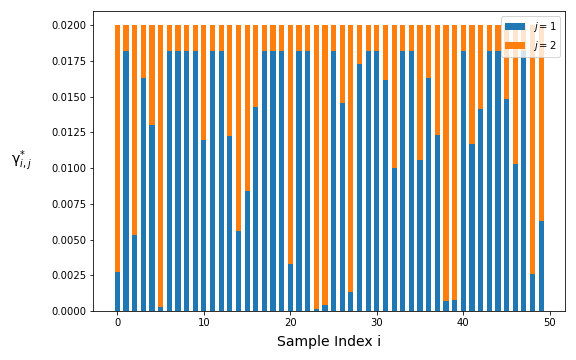}
    \caption{Optimal transport plan $\bm{\gamma}^*$}
    \label{fig:gamma_opt}
  \end{subfigure}
  \caption{Visualization of the numerical solution.}
  \label{fig:solution}
\end{figure}


Figure~\ref{fig:solution} presents the optimal decision vector $\mbf{x}^*$ and the corresponding optimal transport plan $\bm{\gamma}^*$, where the horizontal axis represents the sample index.  In Figure~\ref{fig:X_opt}, we plot the optimal allocation $\mbf{x}^*$ alongside the pricing kernel $\mbf{m}$. The optimal decision exhibits an inverse relation to the pricing kernel: $\mbf{x}_i^*$ tends to be larger when $\mbf{m}_i$ is small, and close to zero when $\mbf{m}_i$ is large. This pattern reflects rational asset allocation under the constraint $\mathbb{E}[\mbf{x}^\top \mbf{m}] \le x_0$, where higher investment is allocated to states with lower marginal cost.

In the worst-case scenario characterized by the optimal transport plan $\bm{\gamma}^*$ from the inner infimum, probability mass is increasingly assigned to the unfavorable outcome $y_1 = 0$ (i.e., $j = 1$) in regions where $\mbf{x}_i$ is large. This strategic reallocation lowers the expected utility, aligning with the adversarial structure of the robust optimization framework.

\section{Conclusion}\label{Sec_conclusion}

This paper studies a robust utility maximization problem for intractable claims under distributional ambiguity, where the distribution of the claim is not fully specified and its dependence with the decision variable is largely unknown. We extend the existing framework in two directions: first, by allowing the marginal distribution of the claim to vary within a $\varphi$-divergence ambiguity set, and second, by considering a general (possibly non-additive) utility function beyond the classical additive specification.

To analyze this problem, we adopt a random distributionally robust optimization (RDRO) formulation, which lifts the optimization to the space of joint distributions and provides a convenient representation of the coupling between the decision and the uncertain claim. By combining tools from optimal transport and convex analysis, we establish the existence and uniqueness of the optimal decision and develop a Legendre-Fenchel duality framework that connects the constrained formulation with its penalized counterpart. In addition, we propose a numerical algorithm based on unbalanced optimal transport scaling and projected gradient methods, and illustrate the relationship between the parameters of the constrained and penalized formulations.

Several extensions for future research remain worth exploring. One natural extension is to consider alternative ambiguity sets, such as those based on Wasserstein distance. Another relevant line of research is to incorporate dependence or structural constraints between the decision variable and the claim, which would lead to more refined models of interaction. These extensions may require new analytical tools, as the current approach relies on specific properties of $\varphi$-divergence and the associated optimal transport structure.



\appendix
\section{Assumption validation for the robust utility maximization problem}\label{app:verification}

This appendix verifies that the primal robust utility maximization problem~\eqref{Primalprob} satisfies all assumptions required by our theoretical framework. 


First, Assumptions~\ref{Asp_utility}, \ref{Asp_utility_concave} and \ref{Asp_Utility_strict_concave} are direct from Assumption~\ref{Asp_U1}, and  Assumption~\ref{Asp_coercive} follows from the chioce of $\varphi$-divergence. 

Verifying Assumption~\ref{Asp_mu_weak_convergent} requires a more technical argument, which we provide below.

\begin{proposition}\label{prop:App_mu_weak_convergent}
Under the market model in Section~\ref{subsec:DRO_invest}, Assumption~\ref{Asp_mu_weak_convergent} holds.
\end{proposition}
\begin{proof}
The budget constraint $\E [ \rho X ] \le x$ and Markov's inequality yield
\begin{equation}\label{tight_x}
\mP[X \ge C_1] \le \mP[\rho X \ge C_1 C_2] + \mP[\rho \le C_2]
\le \frac{x}{C_1 C_2} + \mP[\rho \le C_2]
\end{equation}
for any  constants $C_1, C_2 > 0$. For any $\epsilon > 0$, choose $C_2$ such that $\mathbb{P}[\rho \le C_2] \le \epsilon/2$, and set $C_1 = 2x/(\epsilon C_2)$ to ensure $\mathbb{P}[X \ge C_1] \le \epsilon$. Similarly, there exists $C_3 > 0$ such that $\mP[\rho \ge C_3] \le \epsilon$. Hence,
\begin{equation}
\mP\big[(X,\rho) \in [0, C_1] \times [0, C_3]\big] \ge 1 - 2\epsilon,
\end{equation}
showing tightness of the joint distribution of $(X, \rho)$. 
We denote the joint distribution of $(X, \rho)$ as $\hat{\mu}_X$ and $\mu_X$ is the corresponding distribution of $X$. 
Let $\mu_n$ be any sequence in $\mc{A}_{d}$ with joint distributions $\hat{\mu}_n$. 
By Prokhorov’s theorem, a subsequence $\hat{\mu}_{n_k}$ converges to $\hat{\mu}_0$, and its marginal $\mu_{n_k}$ converges to $\mu_0$. As $\E [ \rho X ] \le x$ for $\hat{\mu}_{n_k}$, the same holds for $\hat{\mu}_0$ by weak convergence, thus $\mu_0 \in \mc{A}_{d}$.
\end{proof}

Finally, we verifying Assumption~\ref{Asp_Subdiff_of_sup}.

It suffices to show that $\Gamma_{\mu}$ is weakly compact for any $\mu \in \mc{A}_{d}$. By \cite[Section 6]{villani2008optimal}, there are many distances that could metrize the weak topology on $\mathcal{P}(\mathcal{X} \times \mathcal{Y})$. Since the space is metrizable, compactness sequentially compact under weak convergence is equivalent to sequentially compactness under weak convergence. Then, this result follows once again from Prokhorov’s theorem.

By proposition~\ref{prop:App_mu_weak_convergent}, $\mc{A}_{d}$ is tight. Thus, for $\mu \in \mc{A}_{d}$, for any $\epsilon >0$, there exists a compact set $K_x \in \mc{X}$ s.t. $\mu(K_x) \ge 1 - \epsilon$.
Then, for any $\gamma \in \Gamma_{\mu}$, since $\mc{Y}$ is bounded, thus compact, we have
\begin{equation}
   \gamma(K_x \times \mc{Y}) =  \mu(K_x) \ge 1 - \epsilon,
\end{equation}
Thus, the family $\Gamma_{\mu}$ is uniformly tight.
Moreover, $\Gamma_{\mu}$ is closed, since the mapping $\gamma \rightarrow \pi^{\sharp}_1 \gamma$ is continuous. By Prokhorov’s theorem, we obtain the sequentially compact under weak convergence.

\section{Realization of a coupling on a the probability space}\label{app:equivalence_of_rv_and_measure}

\begin{lemma}[Realization of a coupling on a non-atomic space]
Let $(\mathcal{X},\mathcal{B}(\mathcal{X}))$ and $(\mathcal{Y},\mathcal{B}(\mathcal{Y}))$ be Polish spaces. Let $\mathcal{P}(\mathcal{X})$ and $\mathcal{P}(\mathcal{Y})$ denote the spaces of Borel probability measures on $\mathcal{X}$ and $\mathcal{Y}$ respectively, and let $\Pi(\mu,\nu) \subset \mathcal{P}(\mathcal{X}\times\mathcal{Y})$ be the set of all couplings of $\mu \in \mathcal{P}(\mathcal{X})$ and $\nu \in \mathcal{P}(\mathcal{Y})$. Let $(\Omega,\mathcal{F},\mathbb{P})$ be a non-atomic probability space and $X:\Omega \to \mathcal{X}$ be a random variable with law $\mathbb{P} \circ X^{-1} = \mu$. Then for any $\gamma \in \Pi(\mu,\nu)$, there exists a random variable $Y:\Omega \to \mathcal{Y}$ such that $(X,Y) \sim \gamma$.
\end{lemma}

\begin{proof}
Since $\mathcal{X}$ and $\mathcal{Y}$ are Polish spaces, the disintegration theorem implies that for any $\gamma \in \Pi(\mu,\nu)$, there exists a measurable stochastic kernel $K:\mathcal{X}\times \mathcal{B}(\mathcal{Y}) \to [0,1]$ such that $\gamma(\dd x, \dd y) = K(x, \dd y)\mu(\dd x)$. Because $(\Omega,\mathcal{F},\mathbb{P})$ is a non-atomic probability space, it is a standard probability space and can support a random variable $U_0:\Omega \to [0,1]$ uniformly distributed on $[0,1]$ (with respect to the Lebesgue measure $\lambda$) such that $U_0$ is independent of $X$. 

Furthermore, as $\mathcal{Y}$ is a Polish space, the kernel $K$ admits a Skorokhod representation: there exists a Borel measurable function $F:\mathcal{X} \times [0,1] \to \mathcal{Y}$ such that for each $x \in \mathcal{X}$, the push-forward of $\lambda$ under $F(x, \cdot)$ is $K(x, \cdot)$, i.e., $\lambda(\{u \in [0,1] : F(x,u) \in B\}) = K(x, B)$ for any $B \in \mathcal{B}(\mathcal{Y})$. 

Define the random variable $Y(\omega) := F(X(\omega), U(\omega))$. The measurability of $Y$ follows from the measurability of $F, X$, and $U_0$. For any bounded Borel measurable function $f:\mathcal{X}\times\mathcal{Y} \to \mathbb{R}$, by the independence of $X$ and $U_0$, we have
\begin{equation}
    \mathbb{E}[f(X,Y)] = \mathbb{E}\left[ \int_0^1 f(X, F(X,u)) \dd \lambda(u) \right].
\end{equation}
By the construction of $F$, the inner integral satisfies
\begin{equation}
    \int_0^1 f(X, F(X,u)) \dd \lambda(u) = \int_{\mathcal{Y}} f(X,y) K(X, \dd y).
\end{equation}
Substituting this into the expectation and using the fact that $X$ has law $\mu$, we obtain
\begin{equation}
    \mathbb{E}[f(X,Y)] = \int_{\mathcal{X}} \left( \int_{\mathcal{Y}} f(x,y) K(x, \dd y) \right) \dd \mu(x).
\end{equation}
Finally, by the definition of the disintegration of $\gamma$, this integral is exactly
\begin{equation}
    \mathbb{E}[f(X,Y)] = \int_{\mathcal{X}\times\mathcal{Y}} f(x,y) \dd \gamma(x,y).
\end{equation}
Thus, $(X,Y)$ is a realization of the coupling $\gamma$, which completes the proof.
\end{proof}

\section{Duality without concavity}\label{app:dual_without_concavity}

Based on the first half of the proof for Theorem~\ref{THM_eta_theta_duality}, in which we haven't used the concavity of $U$, we have
\begin{equation*}
    \begin{aligned}
        J^{c,d}(\eta) & =   \sup_{\theta \ge 0}\{ J^{p,d}(\theta) - \theta \eta \} \\
        & \ge J^{p,d}(\theta) - \theta \eta
    \end{aligned}
\end{equation*}
holds for any $\eta \ge 0$.

Thus, 
\begin{equation}\label{eq:app_jpd}
    J^{p,d}(\theta) \le J^{c,d}(\eta) + \theta \eta
\end{equation}
for any $\theta, \eta \ge 0$.

And the equation in \eqref{eq:app_jpd} holds for $\theta = \theta(\eta, \mu^*_{\eta})$ that solves \eqref{eq:luenberger_inner_dual} with $\mu = \mu^*_{\eta}$.

Thus, 
\begin{equation}
    J^{p,d}(\theta) = \inf_{\eta \ge 0} \{ J^{c,d}(\eta) + \theta \eta \}
\end{equation}
for any $\theta \in \{\theta = \theta(\eta, \mu^*_{\eta}) : \eta \ge 0  \}$.
Thus, $J^{p,d}(\theta)$ is concave on this narrowed domain.

\section*{Acknowledgments}
The authors acknowledge the support from the National Natural Science Foundation of China (Grant No.12271290, No.12371477), the MOE Project of Key Research Institute of Humanities and Social Sciences(22JJD910003). The authors also thank the members of the group of Mathematical Finance and Actuarial Sciences at the Department of Mathematical Sciences, Tsinghua University for their feedback and useful conversations. The authors used ChatGPT (GPT-4, OpenAI, used in June 2025) to assist in improving the grammar and writing style of this manuscript.

\bibliographystyle{apalike}
\bibliography{refs}
\end{document}